\documentclass[10pt]{amsart}
\usepackage{amsmath, amsthm, amsfonts, amssymb}
\usepackage{mathrsfs}
\providecommand{\noopsort[1]{}}
\usepackage{bbm}
\usepackage{dsfont}
\usepackage{stmaryrd}
\usepackage{ifthen}
\numberwithin{equation}{section}
\usepackage{a4}
\usepackage[active]{srcltx}
\usepackage{verbatim}
\usepackage{hyperref}
\allowdisplaybreaks

\newtheorem{thm}{Theorem}[section]
\newtheorem{cor}[thm]{Corollary}
\newtheorem{prop}[thm]{Proposition}
\newtheorem{lem}[thm]{Lemma}

\theoremstyle{remark}
\newtheorem{rem}[thm]{Remark}
\newtheorem{example}[thm]{Example}

\newtheorem{hyp}[thm]{Hypothesis}

\theoremstyle{definition}
\newtheorem{defn}[thm]{Definition}
\renewcommand{\div}{\mathrm{div}\,}
\newcommand{\eps}{\varepsilon}
\newcommand{\weak}{\rightharpoonup}
\newcommand{\one}{\mathbbm{1}}

\newcommand{\bx}{\mathbf{x}}
\newcommand{\by}{\mathbf{y}}
\newcommand{\bw}{\mathbf{w}}

\newcommand{\tn}{t\wedge \tau_n}

\newcommand{\ip}[2]{\ifthenelse{\equal{#1}{}}{\mbox{$ [ \,\cdot\, , \, \cdot \, ] $}}{
\mbox{$ \left[ #1 \, , \, #2 \right]$}}}
\newcommand{\norm}[1]{\ifthenelse{\equal{#1}{}}{\mbox{$\|\cdot\|$}}{\mbox{$\| #1 \|$}}}

\newcommand{\dual}[2]{\ifthenelse{\equal{#1}{}}{\mbox{$ \langle \,\cdot\; , \; \cdot \, \rangle $}}{
\mbox{$ \langle #1   ,  #2 \rangle$}}}

\newcommand{\CR}{\mathds{R}}

\newcommand{\CN}{\mathds{N}}
\newcommand{\CQ}{\mathds{Q}}
\newcommand{\CB}{\mathds{B}}

\newcommand{\ws}{$\mathrm{weak}^*$}

\newcommand{\PP}{\mathbf{P}}
\newcommand{\QQ}{\mathbf{Q}}
\newcommand{\FF}{\mathds{F}}
\renewcommand{\P}{\mathbb{P}}

\newcommand{\bX}{\mathbf{X}}

\newcommand{\bM}{\mathbf{M}}

\newcommand{\cF}{\mathscr{F}}
\newcommand{\cL}{\mathscr{L}}

\newcommand{\cG}{\mathscr{G}}

\newcommand{\cP}{\mathscr{P}}
\newcommand{\cB}{\mathscr{B}}

\newcommand{\expect}{\mathbb{E}}

\newcommand{\half}{\frac{1}{2}}

\renewcommand{\P}{{\mathbb P}}

\newcommand{\OO}{\mathscr{O}}

\begin{document}
\title[On a class of Martingale Problems on Banach Spaces]{On a Class of Martingale Problems on Banach Spaces}
\author{Markus C.\ Kunze}
%\address{Delft Institute of Applied Mathematics, Delft University of Technology, P.O. Box 5031, 2600 GA Delft, 
%The Netherlands}
\address{Institute of Applied Analysis, University of Ulm, 89069 Ulm, Germany}
\email{markus.kunze@uni-ulm.de}
\thanks{The author was supported by VICI subsidy 639.033.604 in the `Vernieuwingsimpuls' program of the Netherlands Organization
for Scientific Research (NWO)}    

\begin{abstract}
We introduce the local martingale problem associated to semilinear stochastic evolution equations driven by a cylindrical Wiener process
and establish a one-to-one correspondence between solutions of the martingale problem and (analytically) weak solutions of the stochastic equation.
We also prove that the solutions of well-posed equations are strong Markov processes.
We apply our results to semilinear stochastic equations with additive noise where the semilinear term is merely measurable and to stochastic 
reaction-diffusion equations with H\"older continuous multiplicative noise.
\end{abstract}
\maketitle

\section{Introduction}
One of the most important tools in the study of stochastic differential equations is the theory of associated martingale problems of Stroock and 
Varadhan \cite{sv1}. At the heart of their approach is the equivalence between solutions of stochastic differential equations (i.e.\ 
stochastic processes) and solutions of the associated martingale problem (i.e.\ probability measures on a function space).

This equivalence is helpful in several ways. First, it can be used to prove \emph{existence of solutions} to stochastic 
differential equations by means of approximation and tightness arguments. Second, it plays an important role in proving 
\emph{uniqueness of solutions} using techniques from semigroup theory or partial differential equations. Last but not least, 
the approach of Stroock and Varadhan yields, given existence and uniqueness of solutions, the strong Markov property of 
the solutions. This plays an important role in the study of further properties of the solutions, e.g.\ their asymptotic behavior.\medskip

In this article, we set up a theory of (local) martingale problems for stochastic evolution equations 
\begin{equation}\label{eq.sde}
 dX(t) = \big[ AX(t) + F(X(t))\big]dt + G(X(t))dW_H(t)\,,
\end{equation}
on a separable Banach space $E$.
Here, $A$ is the generator of a strongly continuous semigroup $S$ on
$E$, $W_H$ is an $H$-cylindrical Wiener process where $H$ is a separable Hilbert space and the 
nonlinearities $F: E \to E$ and $G: E \to \cL (H,E)$ satisfy suitable measurability and (local) 
boundedness assumptions. In fact, we shall consider a slightly more general situation and allow the nonlinearities 
to take values in a larger Banach space $\tilde{E}$, resp.\ $\cL (H, \tilde{E})$.
We will make our assumptions precise in Section \ref{sect.sde}.

Martingale problems for equations of this form on 2-smoothable Banach spaces were studied by Ondrej\'at \cite{ond05}. The usual solution concept 
for equations of the form \eqref{eq.sde} is that of a mild solution which involves a stochastic convolution term. We note that to assure that this term is well-defined, one has to impose additional assumptions on the Banach space (typically geometric assumptions such as the UMD property or 2-smoothability) and/or the coefficients. This poses problems when extending the theory to general Banach spaces.  Here, we overcome these problems 
by basing our theory on (analytically) weak solutions rather than on mild solutions.

Our approach does not only allow us to consider general Banach spaces, it also allows us to work without additional technical 
assumptions (such as the J-property in \cite{ond05}) to ensure stochastic integrability of the occurring processes and to impose only 
minimal assumptions on the coefficients.

Under these minimal assumptions, we introduce the local martingale problem associated to equation \eqref{eq.sde} in Section 
\ref{sect.sde} and establish a one-to-one correspondence between solutions of the local martingale problem and solutions of the 
stochastic evolution equation in Theorem \ref{t.weakmart}. In Theorem \ref{t.wellposed} we prove, given existence and uniqueness of 
solutions, the strong Markov property for solutions of \eqref{eq.sde}, using some abstract results about local martingale problems 
presented in Section \ref{sect.martingale}.

Thus, Sections \ref{sect.martingale} -- \ref{sect.markov} contain the abstract theory of martingale problems on Banach spaces.
In Sections \ref{sect.yw} and \ref{sect.integration} we discuss related results, which we believe are helpful to apply the theory.

In Section \ref{sect.yw} we extend the Yamada-Watanabe theory \cite{yw1} to the setting of Banach spaces and prove that 
pathwise uniqueness implies uniqueness in law (this is the uniqueness concept used in the abstract theory above) and strong existence of 
solutions. As in finite dimensions, pathwise uniqueness can be much easier verified than uniqueness in law in certain situations, in particular 
for equations with (locally) Lipschitz continuous coefficients.

In Section \ref{sect.integration} we show that (analytically) weak and mild solutions coincide if either the coefficient $G$ is constant, i.e.\ in 
equations with additive noise, or if the Banach space $E$ is a UMD space. Working with mild solutions is especially helpful to prove 
existence of solutions, as the standard approach via approximation and tightness often uses the factorization method of \cite{dpkz} as a tool, which, 
in turn, requires a Banach space valued stochastic integral. Here, we use the Banach space valued Wiener integral, see \cite{vNW05}, 
in the case of constant $G$ and the theory of integration in UMD Banach spaces \cite{vNVW07} in the second case. Note that this is the only section 
where we make use of a stochastic integral, all our abstract results do not depend on geometric assumptions on $E$. 
\medskip 

Let us close this introduction by discussing applications of our theory to concrete stochastic evolution equations.  Techniques inspired by 
martingale problems can be found frequently in the literature on infinite dimensional stochastic equations even though, more often than not, 
a martingale problem is not used directly. This is most apparent in the term \emph{martingale solution} which in infinite dimensions does not refer 
to solutions of the martingale problem but is  used synonymously for stochastically weak solutions (thus for stochastic processes). Such solutions were constructed, for example, in \cite{cmg95, gg94a, bg99, zimmer}. Concerning uniqueness,  several authors \cite{cmg95, gg94, zambotti} have proved uniqueness in law for certain equations by using partial differential equations on Hilbert spaces. 

Naturally, the results contained in this article can be used to prove, given well-posedness, the strong Markov property for solutions 
of stochastic evolution equations in arbitrary separable Banach spaces. However, the results obtained here can also be used to \emph{establish}
well-posedness of a given equation. Naturally, the proof of well-posedness of a stochastic evolution equation 
requires additional arguments which depend on the equation in question. 
Thus, the full proofs of our applications to stochastic evolution equations will be given elsewhere \cite{k13, k12}. We will, however, give a rough sketch  in Section \ref{sect.applications} and discuss how the results of this article enter the arguments.

\section{Markov processes and local Martingale Problems}\label{sect.martingale}

In this section $(E, d)$ is a complete, separable metric space. We denote the Borel $\sigma$-algebra of $E$ by
$\cB (E)$. The spaces of scalar-valued measurable, bounded measurable, continuous and bounded continuous functions
will be denoted by $B(E),$ $B_b(E),$ $C(E)$ and $C_b(E)$ respectively. $\cP (E)$ denotes the set of all 
probability measures on $(E, \cB (E))$. For $x \in E$, the Dirac measure in $x$ is denoted by $\delta_x$.

By $C([0,\infty ); E)$ we denote the space of all continuous, $E$-valued functions. The elements
of $C([0,\infty ); E)$ will be denoted by bold lower case letters: $\bx, \by, \mathbf{z}$. 
Endowed with the metric $\pmb{\delta}$, defined by
\[ \pmb{\delta} (\bx , \by ) := \sum_{k=1}^\infty 2^{-k}\sup_{t \in [0,k]}d(\bx_t, \by_t)\wedge 1 , \]
$C([0,\infty ); E)$ is a complete, separable metric space in its own right. We denote its Borel $\sigma$-algebra
by $\cB$. It is well-known that $\cB = \sigma (\bx_s\, : \, s \geq 0)$, see \cite[Lemma 16.1]{kallenberg}. Here,
in slight abuse of notation, we have identified $\bx_s$ with the $E$-valued map $\bx \mapsto \bx_s$.
We shall do so in what follows without further notice.
The filtration generated by these `coordinate mappings' is denoted by $\CB := (\cB_t)_{t\geq 0}$, i.e.\ 
$\cB_t := \sigma (\bx_s\, : \, s \leq t )$. 

The space $\cP (C([0,\infty ); E))$ of probability measures on the Borel $\sigma$-algebra of $C([0, \infty); E)$ will be topologized by the 
\emph{weak topology}, i.e.\ the coarsest topology for which for all bounded continuous function $\Phi$ on $C([0,\infty ); E)$
the map $\PP \mapsto \int \Phi \, d\PP$ is continuous. It is well known that this topology
is metrizable through a complete, separable metric, see \cite[Section II.6]{partha}, i.e.\ $\cP(C([0,\infty ); E))$ 
is a Polish space.

A probability measure $\PP$ on $(C([0,\infty ); E), \cB )$ is 
called a \emph{Markov measure} if the coordinate process $(\bx_t)_{t\geq 0}$  defined on
$(C([0,\infty ); E), \cB, \PP)$ is a Markov process with respect to $\CB$, i.e.\ for all $f \in B_b(E)$ and
$s, t \geq 0$ we have
\[ \expect \big[ f(\bx_{t+s})\big| \cB_t \big] = \expect \big[ f(\bx_{t+s}) \big| \bx_t \big] \quad 
 \PP-a.e.,
\]
where $\expect$ denotes (conditional) expectation with respect to $\PP$.
If this equation also holds whenever $t$ is replaced with a $\CB$-stopping time $\tau$ which is almost surely finite,  i.e.\ 
the coordinate process is a strong Markov process with respect to $\CB$, 
then $\PP$ is called a \emph{strong Markov measure}. Here, as usual, $\cB_\tau$ is the $\sigma$-algebra
\[
\cB_\tau := \{ A \in \cB\,:\, A\cap \{ \tau \leq t \} \in \cB_t\,\, \mbox{for all}\, t \geq 0\}.
\]

A \emph{transition semigroup} is a family $\mathscr{T} := (\mathscr{T}(t))_{t \geq 0}$ of positive contractions on 
$B_b(E)$ such that
\begin{enumerate}
 \item $\mathscr{T}$ is a semigroup, i.e.\ $\mathscr{T}(0) = I$ and $\mathscr{T}(t+s) = \mathscr{T}(t)\mathscr{T}(s)$ for all $t,s \geq 0$.
 \item Every operator $\mathscr{T}(t)$ is associated with a \emph{Markovian kernel}, i.e a map 
$p_t : E\times \cB (E) \to [0,1]$ such that (i) $p_t (x, \cdot ) \in \cP(E)$ for all $x \in E$ and (ii)
$p_t (\cdot , A) \in B_b(E)$ for all $A \in \cB (E)$. That $\mathscr{T}(t)$ is associated with $p_t$ means that
$\mathscr{T}(t)f(x) = \int_E f(y) \, p_t (x, dy)$ for all $f \in B_b(E)$.
\end{enumerate}
The kernels $p_t$ themselves are referred to as \emph{transition functions} or \emph{transition probabilities}.
The semigroup property above is equivalent with the \emph{Chapman-Kolmogorov equations}.

A probability measure $\PP$ on $C([0,\infty ); E)$ is called \emph{Markov measure with transition semigroup $\mathscr{T}$}
if for all $f \in B_b(E)$ and $s,t \geq 0$ we have
\[ \expect \big[ f(\bx_{t+s}) \big| \cB_t \big] = \expect \big[ f (\bx_{t+s}) \big| \bx_t \big] 
 = \big[ \mathscr{T}(s)f\big] (\bx_t) \quad \PP-a.e.
\]
If this equation also holds whenever $t$ is replaced with an $\PP$-a.s. finite $\CB$-stopping time $\tau$, 
then $\PP$ is called a \emph{strong Markov measure with transition semigroup $\mathscr{T}$}.
\medskip

The connection between martingale problems and Markovian measures is well established, see 
\cite[Chapter 4]{ek}. However, if we want to treat stochastic evolution 
equations on Banach spaces, we have to consider \emph{local} martingale problems rather than martingale problems.

\begin{defn}
An \emph{admissible operator} is a map $\cL$, defined on a subset $D(\cL) \subset C(E)$ and taking values
in $B(E)$ such that for all $f \in D(\cL )$ the function $\cL f$ is bounded on compact subsets of $E$. 

Given an admissible operator $\cL$, a probability measure $\PP$ on $C([0,\infty );E)$ is said
to \emph{solve the local martingale problem for $\cL$} if for every $f \in D(\cL )$ the process
$\bM^f$ defined by
\[ \big[\bM^f(\bx )\big](t) := f(\bx_t) - f(\bx_0) - \int_0^t \cL f(\bx_s) \, ds \]
is a local martingale under $\PP$. This of course means that there exists a sequence $\tau_n$, which may depend on 
$f$, of $\CB$-stopping times with $\tau_n \uparrow \infty$ $\PP$-almost surely such that the stopped processes
$\bM_{\tau_n}^f$, defined by $\bM^f_{\tau_n}(t) := \bM^f (t\wedge \tau_n )$, are martingales
for all $n \in \CN$. 

If an initial distribution $\mu \in \cP (E)$ is specified, we say that $\PP$ is a 
\emph{solution to the local martingale problem for} $(\cL, \mu )$ to indicate that in addition to being
a solution to the local martingale problem for $\cL$, the measure $\PP$ satisfies $\PP (\bx_0 \in \Gamma )
= \mu (\Gamma )$ for all $\Gamma \in \cB (E)$, i.e.\ under $\PP$ the random variable $\bx_0$
has distribution $\mu$.
\end{defn} 

We note that by the continuity of $t \mapsto \bx_t$ and since $\cL f$ is bounded on compact subsets
of $E$, the process $\bM^f$ is well-defined. In fact, since $f$ is a continuous function, it follows that
$\bM^f$ is a continuous process. 

The proofs of our results  in Section \ref{sect.markov} are based on the following 
theorem. 

\begin{thm}\label{t.markov}
Let $\cL$ be admissible. Suppose that for every $\mu \in \cP (E)$ any
two solutions $\PP, \QQ$ of the local martingale problem for $(\cL, \mu )$ 
have the same one-dimensional distributions, i.e.\ for all $t \geq 0$ we have
\[ \PP (\bx_t \in \Gamma ) = \QQ (\bx_t \in \Gamma ) \quad \forall \, \Gamma \in \cB (E)\,.\] 
Then 
\begin{enumerate}
\item Every solution of the local martingale problem for $\cL$ is a strong Markov measure.
\item For every $\mu \in \cP (E)$, there is at most one solution to the local martingale 
problem for $(\cL, \mu )$.
\end{enumerate}
If in addition to the uniqueness assumption above for every $x \in E$ there exists a solution $\PP_x$ to the local martingale
problem for $(\cL, \delta_x )$ and if the map $x \mapsto \PP_x(B)$ is Borel measurable for all $B \in \cB$,
then
\begin{itemize}
 \item[(3)] For every $\mu\in \cP (E)$, there exists a solution $\PP_\mu$ of the local 
martingale problem for $(\cL, \mu )$.
 \item[(4)] Define the operator $\mathscr{T}(t)$ by $\mathscr{T}(t)f(x) := \int f(\bx_t )\, d\PP_x$ for $f \in B_b(E)$. 
Then every solution $\PP$ of the local martingale problem for $\cL$ is a strong Markov measure with transition 
semigroup $\mathscr{T} := (\mathscr{T}(t))_{t\geq 0}$.
\end{itemize}
\end{thm}

\begin{proof}
This  Theorem is a generalization of \cite[Theorem 4.4.2]{ek} to local martingale problems. Hence, we have the added difficulty that 
in the definition of ``solution of the local martingale problem'' a sequence of stopping times appears. We only give the proof of 
statement (1), the other statements are derived following the proofs of the corresponding statements in \cite[Theorem 4.4.2]{ek}  with similar
changes due to the presence of stopping times.\smallskip

Let $\PP$ be a solution of the local martingale problem for $(\cL, \mu )$. We denote (conditional)
expectation with respect to $\PP$ by $\expect$. Let $\rho$ be a stopping time with $\rho < \infty$ almost surely and define the mappings
$\Theta_\rho$ and $\Psi_\rho : C([0,\infty ); E) \to C([0,\infty ); E)$ by
\[ (\Theta_\rho\bx )(t) := \bx (t+\rho (\bx)) \quad \mbox{and} \quad (\Psi_\rho \bx) (t) := 
 \bx ( ( t -\rho (\bx))^+ ) \,.
\]
Then $\Theta_\rho$ and $\Psi_\rho$ are measurable mappings with $\Psi_\rho \Theta_\rho\bx = \bx$ for
all $\bx \in C([0,\infty ); E)$.

Now fix $A \in \cB_\rho$ with $\PP(A) >0$ and define the measures $\PP_1, \PP_2$ on $C([0,\infty ); E)$ by
\[
\PP_1 (B ) :=  \frac{ \expect \big[ \one_A \expect [ \one_{\Theta_\rho^{-1}B}| \cB_\rho]\big]}{\PP(A)}
\quad\mbox{and}\quad 
\PP_2 (B ) :=  \frac{ \expect \big[ \one_A \expect [ \one_{\Theta_\rho^{-1}B}| \bx (\rho) ]\big]}{\PP(A)} \,.
\]
We note that under $\PP_1$ and $\PP_2$ the distribution of $\bx (0)$ are identical, namely
for $\Gamma \in \cB (E)$ we have 
\[ \PP_1(\bx (0) \in \Gamma ) = \PP_2(\bx (0) \in \Gamma ) = \PP (\bx (\rho) \in \Gamma | A )\,. \]
Hence, if we prove that $\PP_1$ and $\PP_2$ solve the local martingale problem associated with
$\cL$, we can conclude from our assumption that $\PP_1$ and $\PP_2$ have the same one-dimensional distributions.
This will then imply that for $t >0$ and $\Gamma \in \cB (E)$, we have
\begin{eqnarray*}
&&  \PP_1 (\bx (t) \in \Gamma ) =  \PP (A)^{-1} \expect \big[ \one_A \expect [ \bx (t+\rho ) \in \Gamma |
\cB_\rho ]\big]\\
& = & \PP_2 (\bx (t) \in \Gamma )  =   \PP (A)^{-1} \expect \big[ \one_A \expect [ \bx (t+\rho ) \in \Gamma |
\bx (\rho ) ]\big] \,.
\end{eqnarray*}
Multiplying with $\PP (A)$ and observing that $A$ with $\PP (A)>0$ was arbitrary, it follows that
$\expect [ \bx (t+\rho ) \in \Gamma | \cB_\rho ] = \expect [ \bx (t+\rho ) \in \Gamma | \bx (\rho) ]$.
Since $t, \rho$ and $\Gamma$ were arbitrary, this proves that $(\bx(t))_{t\geq 0}$ is a strong Markov process under
$\PP$.\smallskip

It remains to prove that $\PP_1$ and $\PP_2$ solve the local martingale problem associated with
$\cL$. Fix $f \in D(\cL)$. Since $\PP$ solves the local martingale problem, there exists a sequence $\tau_n$ of stopping times with
$\tau_n \to \infty$ almost everywhere with respect to $\PP$ such that $\bM^f_{\tau_n}$ is a
martingale under $\PP$. We put $\sigma_n := \tau_n\circ \Psi_\rho$. Note that
$\{ \sigma_n \leq t \} = \Psi_\rho^{-1}\{\tau_n \leq t \} \in \cB_t$, since $\tau_n$ is a stopping time
and since $\Psi_\rho^{-1}A \in \cB_t$ for all $A \in \cB_t$, as is easy to see. Hence $\sigma_n$ is a stopping 
time. Since $\Psi_\rho \Theta_\rho\bx = \bx$, it follows from the definition of $\PP_1$ and $\PP_2$
that $\sigma_n \uparrow \infty$ almost surely with respect to $\PP_1$ and $\PP_2$. 

Now fix $t>s$ and $C \in \cB_s$ and observe that
\[
\xi (\bx) := \Big[ \big(\bM^f_{\sigma_n}(t) - \bM^f_{\sigma_n}(s)\big)\one_C\big](\Theta_\rho \bx )
=  \Big[ \big(\bM^f_{\tau_n}(t+\rho) - \bM^f_{\tau_n}(s+\rho)\big)\one_{\Theta_\rho^{-1} C}\big](\bx )
\]
where $\Theta_\rho^{-1}C \in \cB_{s+\rho}$. Since $\bM^f_{\tau_n}$ is a continuous $\PP$-martingale, 
it follows from the optional sampling theorem that
$\expect [\xi | \cB_\rho ] = 0$, and hence, since $\sigma (\bx (\rho)) \subset \cB_\rho$, also 
$\expect [\xi | \bx (\rho)] = 0$. Recalling the definition of $\PP_1$ and $\PP_2$, we see that
that $\bM^f_{\sigma_n}$ is a martingale under $\PP_1$ and $\PP_2$. 
\end{proof}

\begin{defn}
Let $\cL$ be an admissible operator. We say that the local martingale problem for $\cL$
is \emph{well-posed} if for every $x \in E$, there exists a unique solution $\PP_x$
of the local martingale problem for $(\cL, \delta_x)$. 

We say that the martingale problem for $\cL$ is \emph{completely well-posed}, if (i) for every $\mu \in \cP (E)$
there exists a unique solution $\PP_\mu$ of the local martingale problem for $(\cL, \mu )$ and (ii)
the map $x \mapsto \PP_x(B)$ is measurable for every $B \in \cB$.  

In the case of uniqueness, we will use the notation $\PP_x$ resp.\ $\PP_\mu$ for the solution of
the local martingale problem for $(\cL, \delta_x)$, resp.\ $(\cL, \mu )$.
\end{defn}

In Theorem \ref{t.wellposed}, we will prove that if the martingale problem for $\cL$ is well-posed, then it is already 
completely well-posed. Thus, we obtain the measurability of the map $x \mapsto \PP_x$ and existence and uniqueness of 
solutions for arbitrary initial distributions $\mu$ for free.

We note that by (2) of Theorem \ref{t.markov}, the uniqueness assumption in the definition 
of `completely well-posed' can be weakened to uniqueness of the one-dimensional marginals. Similarly, 
by (3) of Theorem \ref{t.markov}, in the definition of `completely well-posed' it suffices to assume 
existence of solutions only for degenerate initial distributions $\delta_x$, for all $x \in E$.

By part (4) of Theorem \ref{t.markov}, if the local martingale problem for $\cL$ is completely well-posed, then there exists a 
transition semigroup $\mathscr{T}$ such that every solution $\PP_\mu$ is a strong Markov measure with transition semigroup 
$\mathscr{T}$. This semigroup $\mathscr{T}$ is uniquely determined by $\cL$ and will be called the \emph{associated semigroup}.

\section{Stochastic differential equations and the associated local martingale problem}\label{sect.sde}

We now turn our attention to the stochastic evolution equation \eqref{eq.sde}. In order to stress the dependence 
on the coefficients, we will also refer to equation \eqref{eq.sde} as equation $[A,F,G]$.
The following are our standing hypotheses on the coefficients and will be assumed in the rest of this paper.

\begin{hyp}\label{hyp1}
$\tilde{E}$ is a separable Banach space and $A$ generates a strongly continuous semigroup $S := (S(t))_{t\geq 0}=
(S_t)_{t\geq 0}$ on $\tilde{E}$.
$H$ is a separable Hilbert space and $W_H$ is an $H$-cylindrical Wiener process. 
$E$ is a separable Banach space such that $D(A)\subset E\subset \tilde{E}$ 
with continuous and dense embeddings. Throughout, all Banach spaces are real. Furthermore,
\begin{enumerate}
 \item $F: E \to \tilde{E}$ is strongly measurable and bounded on bounded subsets of $E$;
 \item $G: E \to \cL (H,\tilde{E})$ is $H$-strongly measurable, i.e.\ $Gh : E \to \tilde{E}$ is strongly measurable for 
all $h \in H$, and $G$ is bounded on bounded subsets of $E$.
\end{enumerate}
\end{hyp}

\begin{example}
Let us describe typical examples in which Hypothesis \ref{hyp1} is satisfied.

In the easiest example, $\tilde{E} = E$ and $A$ is the generator of a strongly continuous $S$ on $\tilde{E}$. 
In applications, $A$ is typically a differential operator and $\tilde{E}$ is an $L^p$-space.
In that situation, it is also possible to replace $E$ with a suitable Sobolev space or a space of continuous functions.
To model equations driven by (additive or multiplicative) white noise, it is often useful to replace $\tilde{E}$ with 
a suitable extrapolation space, see, for example, \cite{vNVW08}.

In these situations, the semigroup $S$ typically maps $\tilde{E}$ into $E$ and restricts to a strongly continuous 
semigroup on $E$. Moreover, one has some control over the norms $\|S(t)\|_{\cL (\tilde{E}, E)}$ at $t=0$. It should be noted,
that we assume none of this in Hypothesis \ref{hyp1}. However, later on (in Hypothesis \ref{hyp2}) we will make precisely 
these assumptions. 
\end{example}

Before defining what we mean by `a solution' of equation $[A,F,G]$,
let us recall the notion of an $H$-cylindrical Wiener process. Let $(\Omega, \Sigma, \FF, \P)$ be a stochastic basis, i.e.\ a 
probability space $(\Omega, \Sigma, \P)$ together with a filtration $\FF = (\cF_t)_{t\geq 0}$. We say that the \emph{usual conditions}
are satisfied if $\cF_0$ contains all $\P$-null sets and the filtration is right continuous.

An \emph{$H$-cylindrical Wiener process (with respect to $\FF$)}
 is a bounded linear operator $W_H$ from $L^2(0,\infty; H)$ to $L^2(\Omega, \Sigma , \P)$ with the following properties:
\begin{enumerate}
 \item for all $f \in L^2(0,\infty; H)$ the random variable $W_H(f)$ is centered Gaussian.
 \item for all $t\geq 0$ and $f \in L^2(0,\infty; H)$ with support in $[0,t]$, the random variable $W_H(f)$ is $\cF_t$-measurable.
 \item for all $t\geq 0$ and $f \in L^2(0,\infty; H)$ with support in $[t, \infty)$, the random variable $W_H(f)$ is independent of $\cF_t$.
 \item for all $f_1,f_2\in L^2(0,\infty; H)$ we have $\expect (W_H(f_1)W_H(f_2)) = [f_1,f_2]_{L^2(0,\infty; H)}$.
\end{enumerate}

We shall write 
\[ W_H(t)h := W_H(\one_{(0,t]}\otimes h), \quad t>0, h \in H.\]
It is easy to see that for $h \in H$ the process $W_Hh := (W_H(t)h)_{t \geq 0}$ is a real-valued Brownian motion 
(which is standard if $\|h\|_H = 1$).

We now define the concept of a weak solution. The relation of weak solution with other solution concepts 
will be discussed in Section \ref{sect.integration}.

\begin{defn}
A tuple $\big( (\Omega, \Sigma , \FF, \P),  W_H, \bX\big)$, where $(\Omega, \Sigma , \FF, \P)$ is stochastic basis satisfying the usual conditions, $W_H$ is an H-cylindrical Wiener process with respect to $\FF$ and
$\bX = (X_t)_{t\geq 0}$ is a continuous, $\FF$-progressive, $E$-valued process is called \emph{weak solution} of \eqref{eq.sde} if
for all $x^* \in D(A^*) \subset\tilde{E}^*$ and $t \geq 0$ we have
\begin{equation}\label{eq.sol}
  \dual{X_t}{x^*} = \dual{X_0}{x^*} + \int_0^t \dual{X_s}{A^*x^*}\, ds 
 + \int_0^t \dual{F(X_s)}{x^*}\, ds + \int_0^t G(X_s)^*x^*\, dW_H(s)\,,
\end{equation}
$\P$-a.e.
\end{defn}

\begin{rem}
Weak solutions are weak both in the analytic sense, i.e.\ we require \eqref{eq.sol} to hold only if tested 
against functionals $x^* \in D(A^*)$ and in the probabilistic sense, i.e.\ the stochastic basis and the cylindrical Wiener process are part of the solution. More appropriately, we should speak of `analytically weak and stochastically weak solution' or `weak martingale solution'.
However, to shorten notation, we have settled on the term `weak solution'.
\end{rem}

By the continuity of the paths and our assumptions in Hypothesis \ref{hyp1}, the Lebesgue-integral
in \eqref{eq.sol} is well defined. The stochastic integral in
equation \eqref{eq.sol} is an integral of an $H \simeq H^*$-valued stochastic
processes with respect to a cylindrical Wiener process. It is well known how to construct such an
integral for progressive $H$-valued processes $\Phi$ such that $\Phi \in L^2(0,T;H)$ almost surely for all $T >0$. 
Namely, if $(h_k)$ is a (finite or countably infinite) orthonormal basis of the separable Hilbert space $H$ and we define
$\beta_k(s) := W_H(s)h_k$, then 
\[
\int_0^t \Phi(s) \, dW_H(s) := \sum_k \int_0^t \ip{\Phi(s)}{h_k}_H\, d\beta_k(s)\,.
\]
The integral process $\mathbf{I}(t) := \int_0^t \Phi (s)dW_H(s)$ is a real-valued, continuous, local martingale with 
with quadratic variation $\llbracket \mathbf{I} \rrbracket_t = \int_0^t \norm{\Phi (s)}^2_H\, ds$. We also note that
for an $\FF$-stopping time $\tau$ we have almost surely
$\mathbf{I}(t\wedge \tau ) = \int_0^t \one_{[0,\tau]}(s)\Phi (s) \, dW_H(s)$ for all $t \geq 0$.

In order to shorten notation, we will say that a process $\bX$ is a weak solution of \eqref{eq.sde}, meaning
that $\bX$ is a continuous, progressive, $E$-valued process, defined on a stochastic basis $(\Omega, \Sigma, \P, \FF)$, satisfying the usual 
conditions, on 
which an $H$-cylindrical Wiener process $W_H$ with respect to $\FF$ is defined such that the tupel
$((\Omega, \Sigma, \FF, \P), W_H, \bX)$ is a weak solution of \eqref{eq.sde}. In this case, unless stated otherwise,
$\P$ will denote the measure on the probability space and $W_H$ the $H$-cylindrical Wiener process.
These remarks apply, mutatis mutandis, also for the other solution concepts that we will introduce.

\begin{rem}\label{rem1}
We note that the exceptional set in \eqref{eq.sol} which initially depends on $x^*$ and $t$ may be chosen independently
of $t$, since the deterministic integrals as well as the stochastic integral in \eqref{eq.sol} are pathwise continuous
in $t$.
\end{rem}

We now establish a one-to-one correspondence between weak solutions of equation $[A,F,G]$ and solutions 
of the local martingale problem for an (admissible) operator $\cL_{[A,F,G]}$ which we call the \emph{associated local
martingale problem}.

The operator $\cL_{[A,F,G]}$ is defined as follows.

By $\mathscr{D}$ we denote the vector space of all functions $f: E \to \CR$ of the form
\[ f(x) = \varphi ( \dual{x}{x_1^*}, \ldots, \dual{x}{x_n^*}) \]
where $n \in \CN$, $\varphi \in C^2(\CR^n)$ and $x_1^*, \ldots, x_n^* \in D(A^*)$.

For $f = \varphi ( \dual{\cdot}{x_1^*}, \ldots, \dual{\cdot}{x_n^*}) \in \mathscr{D}$ we put
\begin{equation}\label{eq.l}
\begin{aligned}
L_{[A,F,G]} f (x) := & \sum_{k=1}^n \frac{\partial\varphi}{\partial u_k}(\dual{x}{x_1^*},\ldots, 
 \dual{x}{x_n^*})\cdot \big[ \dual{x}{A^*x_k^*} + \dual{F(x)}{x_k^*}\big]\\
& + \half \sum_{k,l=1}^n [G(x)^*x_k^*\,,\,G(x)^*x_l^*]_H \frac{\partial^2\varphi}{\partial u_k\partial u_l}
(\dual{x}{x^*_1}, \ldots, \dual{x}{x_n^*})
\end{aligned}
\end{equation}

The operator $\cL_{[A,F,G]}$ is defined by $D(\cL) = \mathscr{D}$ and $\cL_{[A,F,G]} f := L_{[A,F,G]}f$. Put 
$\mathscr{D}_{\min}  := \big\{ \dual{\cdot}{x^*}^j\,: \, x^*\in D(A^*), j=1,2\big\}$. We will also use 
the operator $\cL_{[A,F,G]}^{\min} := \cL_{[A,F,G]}|_{\mathscr{D}_{\min}}$.
We note that since $F$ and $G$ are bounded on bounded subsets of $E$, 
the operators $\cL_{[A,F,G]}$ and $\cL_{[A,F,G]}^{\min}$ are admissible. We would like to point out that the function $\cL_{[A, F, G]}f$ even if $\varphi$
has compact support. This is the reason for considering local martingale problems, rather than martingale problems.

\begin{thm}\label{t.weakmart}
Suppose that $\bX$ is a weak solution of equation $[A,F,G]$. Then the law $\PP$ of $\bX$ solves the local martingale 
problem for $\cL_{[A,F,G]}$. 

Conversely, if $\PP$ solves the local martingale problem for $\cL_{[A,F,G]}^\mathrm{min}$, 
then there exists a weak solution $\bX$ of equation $[A,F,G]$ with distribution $\PP$. 
\end{thm}

\begin{proof}
First suppose that $\bX$ is a weak solution of equation $[A,F,G]$.

Let $f =  \varphi ( \dual{\cdot}{x_1^*}, \ldots, \dual{\cdot}{x_n^*}) \in \mathscr{D}$ and define the
$\CR^n$-valued process $\xi$ by $\xi_k(t) = \dual{X(t)}{x_k^*}$ for all $t\geq 0$ and $k=1,\ldots , n$.
We also define $\CR^n$-valued processes $V$ and $M$ by
\[ V_k(t) := \int_0^t \dual{X_s}{A^*x^*_k} + \dual{F(X_s)}{x^*_k}\, ds
 \;\; , \;\; M_k(t) := \int_0^t G(X_s)^*x^*_k\, dW_H(s) ,
\]
for $k=1, \ldots , n$. Note that, almost surely, $V$ has continuous trajectories of locally bounded
variation and that $M$ is a continuous, local martingale. Since $X$ is a weak solution, it follows that 
$\xi = \xi_0 + M +V$.

 It\^o's formula
\cite[Theorem 5.2.9]{ek} yields
\[
 \begin{aligned}
& \quad f(X_t)  - f(X_0)   =  \varphi (\xi_t) - \varphi (\xi_0) \\
 &=   \sum_{k=1}^n \int_0^t \frac{\partial\varphi}{\partial u_k}(\xi_s )\, dV_k(s)
 + \half \sum_{k,l=1}^n \int_0^t \frac{\partial^2\varphi}{\partial u_k\partial u_l}(\xi_s )\, d \llbracket M_k, 
M_l\rrbracket_s\\ & \quad\quad\quad 
+ \sum_{k=1}^n \int_0^t \frac{\partial\varphi}{\partial u_k}(\xi_s )\, dM_k(s)
\\
&=  \int_0^t \big[L_{[A,F,G]} f\big] (X_s)\, ds + \sum_{k=1}^n \int_0^t \frac{\partial\varphi}{\partial u_k}(\xi_s)\, dM_k(s)\, \,,
 \end{aligned}
\]
for all $t \geq 0$.
Here, we have used that $\llbracket M_k, M_l\rrbracket_t = \int_0^t [G(X_s)^*x_k^*,G(X_s)^*x_l^*]_H\, ds$.
It thus follows that
\[ f(X_t) - f(X_0)  - \int_0^t [ L_{[A,F,G]} f ](X_s)\, ds \]
is a continuous local martingale with respect to $\FF$. Passing to the range space $C([0,\infty ); E)$, it
follows that under the distribution $\PP$ of $\bX$, the process $\bM^f$ is a continuous local martingale
with respect to $\CB$.\medskip

We now prove the converse. First note that if $x^* \in D(A^*)$, then for $f_1(x) = \dual{x}{x^*}$ we have
$L_{[A,F,G]} f_1 (x) = \dual{x}{A^*x^*} + \dual{F(x)}{x^*}$ and for $f_2(x) = \dual{x}{x^*}^2$
we have $L_{[A,F,G]} f_2(x) = 2\dual{x}{x^*}\cdot \big[\dual{x}{A^*x^*} + \dual{F(x)}{x^*}\big]
+ \norm{G(x)^*x^*}_H^2$. If $\PP$ is a solution of the local martingale problem for $\cL_{[A,F,G]}$, then under $\PP$
the processes $\bM^{f_1}$ and $\bM^{f_2}$ are
local martingales with respect to the canonical filtration $\CB$. Using that the coefficients $F$ and $G$ are bounded 
on bounded subsets, an approximation argument shows that we can use $\tau_n := \inf\{ t> 0: \| \bx (t) \| \geq n \}$ as 
localizing sequence for both $\bM^{f_1}$ and $\bM^{f_2}$. As in \cite[Chapter 5, Problem 4.13]{ks} we see that 
the stopped processes $\bM^{f_1}_{\tau_n}$ and $\bM^{f_2}_{\tau_n}$ are martingales with respect to filtration 
$\FF := (\cF_t)$, where $\cF_t$ is the augmentation of $\cB_{t+}$ by the $\PP$ null sets. 
Hence $\bM^{f_1}$ and $\bM^{f_2}$ are local martingales with respect 
to the filtration $\FF$, which satisfies the usual conditions.
It now follows from \cite[Lemma 34]{ond05} that under $\PP$ the process
\[ \dual{\bx_t}{x^*} - \dual{\bx_0}{x^*} - \int_0^t \dual{\bx_s}{A^*x^*} +
 \dual{F(\bx_s)}{x^*}\, ds
\]
is a continuous local martingale with quadratic variation $\int_0^t \norm{G(\bx_s)^*x^*}_H^2\, ds$.
By \cite[Theorem 3.1]{ondr1}, we find an extension $(\Omega, \Sigma, \tilde{\FF}, \P)$ of
$(C([0,\infty );E), \mathscr{B}, \FF, \PP)$ on which a cylindrical Brownian motion $W_H$
is defined such that for all $x^* \in D(A^*)$ we have 
\[ \dual{\bx_t}{x^*} - \dual{\bx_0}{x^*} - \int_0^t \dual{\bx_s}{A^*x^*} +
 \dual{F(\bx_s)}{x^*}\, ds = \int_0^t G(\bx_s)^*x^* dW_H(s)
\]
$\PP$-almost everywhere for all $t \geq 0$
This proves that $\bx$, defined on this extension, is a weak solution of $[A,F,G]$.
\end{proof}

\begin{cor}\label{c.llmin}
 A measure $\PP \in \cP (C([0,\infty ); E)$ solves the local martingale problem for $\cL_{[A,F,G]}$ if and 
only if it solves the local martingale problem for $\cL_{[A,F,G]}^{\min}$.
\end{cor}

Motivated by Theorem \ref{t.weakmart} we will say that the local martingale problem for $\cL_{[A,F,G]}$ 
is the local martingale problem \emph{associated with} equation $[A,F,G]$. We will say that equation 
$[A,F,G]$ is \emph{(completely) well-posed} if the associated local martingale problem is (completely) well-posed.

\section{Well-posed equations and the strong Markov property}\label{sect.markov}

In this section we prove that if equation $[A,F,G]$ is well-posed, then it is completely well-posed. The 
results of Section \ref{sect.martingale} then imply that  
solution of $[A,F,G]$ is a strong Markov process with transition semigroup $\mathscr{T} := (\mathscr{T}(t))_{t\geq 0}$, 
where $\mathscr{T}(t)f(x) = \int_E f(\bx_t)\, d\PP_x$. 

The key step in the proof is is to show that it even suffices 
to consider the local martingale problem for an operator $\cL_{[A,F,G]}^0$, defined on a countable set, cf.\ \cite[Theorem 4.4.6]{ek}.

\begin{lem}\label{l.count}
There exists a countable subset $\mathscr{D}_0$ of $\mathscr{D}$ such that
a measure $\PP$ solves the local martingale problem associated for $\cL_{[A,F,G]}$ if and only if it solves the martingale 
problem associated with $\cL_{[A,F,G]}^0:= \cL_{[A,F,G]}|_{\mathscr{D}_0}$.
\end{lem}

\begin{proof}

{\it Step 1:} We construct the set $\mathscr{D}_0$.

First note that there exists a countable subset $D$ of $D(A^*)$ such that for every $x^* \in D(A^*)$
there exists a sequence $(x_n^*) \subset D$ such that $x_n^* \weak^* x^*$ and $A^*x_n^* \weak^* A^*x^*$.
Here $\weak^*$ refers to weak$^*$ convergence in $\tilde{E}^*$.
To see this, first note that there is a countable set $\{z_n^*\, : \, n \in \CN\} \subset \tilde{E}^*$ which is 
sequentially weak$^*$-dense in $\tilde{E}^*$, see \S 21.3 (5) of \cite{koethe}. Put $D := \{ R(\lambda, A^*)z_n^* \, : \,
n\in\CN\}$ for some $\lambda \in \rho (A^*)$. Using that $R(\lambda, A^*)$ is $\sigma (\tilde{E}^*,\tilde{E})$-continuous 
as an adjoint operator, it is easy to see that $D$ has the required properties. Replacing $D$ with the 
set of all convex combinations of elements of $D$ with rational coefficients, we may (and shall) assume that such
convex combinations belong to $D$ again.

Now choose a sequence $\varphi_n \in C^2(\CR )$ with the following properties:
\begin{enumerate}
 \item $\varphi_n(t) = t$ for all $-n \leq t \leq n$ and $\varphi_n(t) = 0$
for $t \not\in [-2n, 2n]$.
 \item $\sup_n\norm{\varphi'_n}_{\infty}, \sup_n\norm{\varphi''_n}_\infty < \infty$.
\end{enumerate}
We then define 
\[ \mathscr{D}_0 := \big\{ f = \varphi_n(\dual{\cdot }{x^*})^j\quad\mbox{for some}\,\,
 n \in \CN\, , \, x^* \in D \, ,\, j \in \{1,2\}\big\} .
\]
Clearly, $\mathscr{D}_0$ is countable.
We define $\cL_{[A,F,G]}^0:= \cL_{[A,F,G]}|_{\mathscr{D}_0}$.\medskip 

{\it Step 2:} Now let $\PP$ be a solution of the local martingale problem for $\cL_{[A,F,G]}^0$. We prove that $\PP$
 solves the local martingale problem for $\cL_{[A,F,G]}^\mathrm{min}$. This finishes the proof in 
view of Corollary \ref{c.llmin}.

First note that $\bM^f$ is a local martingale for any $f= \dual{\cdot}{x^*}^j$, $x^*\in D\,,\, j\in\{1,2\}$.
To see this, let $\sigma_n := \inf\{ t >0 \, : \, |\dual{\bx_t}{x^*}|\vee \|\bx_t\| \geq n \}$ and put 
$f_n := \varphi_n(\dual{\cdot}{x^*})^j \in \mathscr{D}_0$. Clearly, $\bM^f_{\sigma_n} = \bM^{f_n}_{\sigma_n}$. 
Since $\PP$ solves the local martingale problem for $\cL_{[A,F,G]}^0$, the process $\bM^{f_n}$, hence by 
optional sampling also $\bM^{f_n}_{\sigma_n}$, is a local martingale under $\PP$. 
Since $F$ and $G$ are bounded on bounded sets, $\bM^{f_n}_{\sigma_n}$ is uniformly bounded.
Thus, $\bM^{f_n}_{\sigma_n}$ is a true martingale by dominated convergence.
This proves that $\bM^f_{\sigma_n}$ is a true martingale under $\PP$ and hence, 
since $\sigma_n \uparrow \infty$ pointwise, that $\bM^f$ is a local martingale under $\PP$.\medskip 

It remains to extend this from $x^* \in D$ to arbitrary $x^* \in D(A^*)$. To that end,
fix $x^*\in D(A^*)$ and a sequence $(x_n^*)\subset D$ such that $x_n^* \weak^* x^*$
and $A^*x_n^* \weak^* A^*x^*$. By the uniform boundedness principle, 
the sequences $(x_n^*)$ and $(A^*x_n^*)$ are bounded in $\tilde{E}^*$, say by $M$.
For $m \in \CN$ put $\tau_m := \inf\{ t>0 \,:\, \norm{\bx (t)} \geq m\}$.

Let us first consider $f := \dual{\cdot }{x^*}$. Arguing as above, we see that for 
$f_n := \dual{\cdot}{x_n^*}$, the stopped process $\bM_{\tau_m}^{f_n}$ is a martingale under $\PP$ for all 
$n,m \in \CN$. Furthermore, since $L_{[A,F,G]}f_n \to L_{[A,F,G]}f$ pointwise, it follows that 
$\bM_{\tau_m}^{f_n}(t) \to \bM_{\tau_m}^{f}(t)$ pointwise as $n\to \infty$,
for all $t \geq 0$.
Since $F$ is bounded on $\bar{B}(0,m)$, say by $C_m$, we find for $t>s$
\[ \big| \bM^{f_n}_{\tau_m}(\bx)(t) - \bM^{f_n}_{\tau_m}(\bx )(s) \big|
 \leq (t-s) \big[ m\cdot M + C_m\cdot M\big] + 2m\cdot M
\]
for all $n, m \in \CN$. Applying the dominated convergence theorem to the sequence 
$(\bM_{\tau_m}^{f_n}(t) - \bM_{\tau_m}^{f_n}(s))\one_B$, where $B$ is an arbitrary set in $\cB_s$, it follows that 
$\int_B\bM_{\tau_m}^{f}(t) - \bM_{\tau_m}^{f}(s))\, d\PP =0$. Since $0\leq s < t$ and $B \in \cB_s$ were 
arbitrary, $\bM_{\tau_m}^f$ is a $\CB$-martingale 
under $\PP$. As $\tau_m\uparrow \infty$ almost surely, this proves that $\bM^f$ is a local martingale under $\PP$.\smallskip

Next consider $f := \dual{\cdot}{x^*}^2$. For $f_n := \dual{\cdot }{x_n^*}^2$, the stopped process
$\bM_{\tau_m}^{f_n}$ is a martingale under $\PP$ for all $n, m \in \CN$. Similarly as above, one
sees that for every $m \in \CN$ the difference $|\bM_{\tau_m}^{f_n}(t) - \bM_{\tau_m}^{f_n}(s)|$
may be majorized by a bounded function independent of $n$. However, due to the term $\|G(\cdot )^*x_n^*\|_H^2$
in $L_{[A,F,G]}f_n$, the weak convergence $x_n^* \weak^* x^*$ does not suffice to conclude that 
$L_{[A,F,G]}f_n \to L_{[A,F,G]}f$ pointwise.
Hence we employ a different method here. 

We fix $0\leq s < t$ and $m \in \CN$. The dominated convergence theorem yields weak convergence 
\[ 
\int_s^t\one_{[0,\tau_m]}(r) G(\bx_r)^*x_n^*\, dr \weak \int_s^t \one_{[0,\tau_m]}(r) G(\bx_r)^*x^*\, dr
 \quad \mbox{in}\,\, L^2(C([0,\infty);E), \PP; H)\,.
\]
Hence  $\int_s^t\one_{[0,\tau_m]}(r) G(\bx_r)^*x^*\, dr$ belongs to the weak closure of
the tail sequence $\big( \int_s^t\one_{[0,\tau_m]}(r) G(\bx_r)^*x_n^* \, dr\big)_{n\geq N}$, for any $N \in \CN$. 
By the Hahn-Banach theorem, it belongs to the strong closure of that tail, whence we find vectors $y_N^*$, belonging 
to the convex hull the sequence $(x_n^*)_{n\geq N}$, such that we have \emph{strong} convergence 
\[ 
 \int_s^t\one_{[0,\tau_m]}(r) G(\bx_r)^*y_N^*\, dr \to \int_s^t\one_{[0,\tau_m]}(r) G(\bx_r)^*x^*\, dr
 \quad \mbox{in}\,\, L^2(C([0,\infty);E), \PP;H)\,.
\]
After passing to a subsequence, we may assume that this convergence holds pointwise $\PP$-a.e.
Note that $y_N^* \weak^* x^*$, as $y_N^*$ belongs to the tail $(x_n^*)_{n\geq N}$. Hence
it follows that
\[ \bM_{\tau_m}^{g_N}(t) - \bM_{\tau_m}^{g_N}(s) 
 \to \bM_{\tau_m}^{f}(t) - \bM_{\tau_m}^{f}(s)
\]
pointwise $\PP$-almost everywhere. Here, $g_N := \dual{\cdot}{y_N^*}^2$. 

Note that we may assume without loss of generality that $y_N^*$ is a convex combination 
of the $(x_n^*)_{n\geq N}$ with rational coefficients.
Hence, $y_N \in D$ and thus $g_N \in \mathscr{D}_0$, implying that $\bM_{\tau_m}^{g_N}$ is a martingale under $\PP$
for all $N \in \CN$. Now, similarly as above, the dominated convergence theorem shows that $\bM^f_{\tau_m}$ is a martingale
under $\PP$
for all $m \in \CN$. This finishes the proof.
\end{proof}

Now the announced result about the equivalence of well-posedness and complete well-posedness 
follows similar to the finite-dimensional case, cf.\ \cite[Theorem 21.10]{kallenberg}.

\begin{thm}\label{t.wellposed}
Suppose that the local martingale problem for $\cL_{[A,F,G]}$ is well-posed. Then it is completely well-posed. 
Consequently, all weak solutions of equation $[A,F,G]$ are strong Markov processes with a common transition 
semigroup $\mathscr{T}$.
\end{thm}

\begin{proof}
We first prove the measurability of the map $x \mapsto \PP_x$.
Consider the set $V := \{ \PP_x \, : \, x \in E\}$. 
We claim that $V$ is a Borel subset of $\cP (C([0,\infty);E))$. 
Indeed, by well-posedness, $V = V_1\cap V_2$, where $V_1$ is the set of all probability measures with degenerate
initial distributions and $V_2$ is the set of all solutions to the martingale problem. 

Since the map $\PP \mapsto
\PP\circ \bx(0)^{-1}$ is measurable from $\cP (C([0,\infty);E))$ to $\cP (E)$, the measurability
of $V_1$ follows from \cite[Lemma 1.39]{kallenberg}.

By Lemma \ref{l.count}, $\PP \in V_2$ if and only if $\bM^{f}$ is a local martingale under $\PP$ 
for all $f \in \mathscr{D}_0$.
With $\tau_n := \inf\{t >0\,:\, \norm{\bx (t)}\geq n\}$, this is equivalent with 
\[ \int_B \bM^{f}(\tn) \, d\PP = \int_B \bM^{f}(s\wedge \tau_n) \, d\PP \quad \forall\, s< t, \,
 B \in \mathscr{B}_s \,,\, n \in \CN.
\]
However, using continuity of $t \mapsto \bx (t)$ and the fact that the $\sigma$-algebra $\mathscr{B}_s$
is countably generated for all $s >0$, we see that $\bM^f$ is a local martingale under $\PP$ whenever the above equality holds 
for $n \in \CN, s,t \in \CQ$ with $s<t$ and $B$ in a countable subset of $\mathscr{B}_s$. 
Hence the set $V_2$ is determined by countably many `measurable relations' and hence measurable.
It follows that $V$ is measurable as claimed.\smallskip 

Now define the map $\Phi : V \to E$ by defining $\Phi (\PP)$ as the unique $x$ such that 
$\PP\circ \bx_0^{-1} = \delta_x$. Clearly, $\Phi$ is injective. Furthermore, $\Phi$ is measurable as 
the composition of the measurable map $\PP \circ \bx_0^{-1}$ and the inverse of
the map $x \mapsto \delta_x$, which establishes a homeomorphism between $E$ and the range of that map.
By the Kuratowski Theorem, see \cite[Section 1.3]{partha}, the inverse $\Phi^{-1}$ is measurable, i.e.
$x \mapsto \PP_x$ is a measurable map from $E$ to $\cP (C([0,\infty);E))$
\medskip

It remains to prove the uniqueness of solutions with arbitrary initial distributions $\mu$ 
for the martingale problem for $\cL_{[A,F,G]}$. The existence of
solutions with general initial distributions will then follow from Theorem \ref{t.markov}.

To that end, assume that $\PP$ solves the local martingale problem for $\cL_{[A,F,G]}$ and that 
$\bx(0)$ has distribution $\mu \in \cP (E)$.
Let $\QQ : E \times \cB \to [0,1]$ be a regular conditional probability (under $\PP$) 
for $\cB$ given $\bx_0$. Then
\[ \PP (A) = \int_E \QQ (x,A)\, d\mu (x) \quad \forall\, A \in \mathscr{B}\,.\]

Now let $t>s\geq 0$ and $B \in \cB_s$ be given. Then, for $f\in \mathscr{D}$, we have
\[
 \int_B \bM^{f}(\tn) - \bM^{f}(s\wedge \tau_n)\, d\QQ(x, \cdot ) =
\int_{B\cap \{\bx (0) = x\}} \bM^{f}(\tn ) - \bM^{f}(s\wedge \tau_n) \, d\PP = 0
\]
for $\mu$-almost every $x$.
We note that the null-set outside of which this equation holds depends on $t, s, n, B$ and the function $f$. 
However, arguing as above, we see that for fixed $f$, there exists a null-set $N(f)$, such that the above equation holds
outside $N(f)$ for \emph{all} $t>s, n \in \CN$ and $B \in \cB_s$. Putting $N := \bigcup_{f \in \mathscr{D}_0}N(f)$, 
it follows that outside of $N$, the above holds for all $t>s,n\in \CN, B \in \cB_s$ and $f \in \mathscr{D}_0$. 
This implies that for $\mu$-a.e.\ $x$ the measure $\QQ (x, \cdot )$ solves the local martingale problem for 
$\cL_{[A,F,G]}^0$ and hence, by Lemma \ref{l.count}, the local martingale problem for $\cL_{[A,F,G]}$. 
By well-posedness, $\QQ(x, \cdot ) = \PP_x (\cdot )$ for $\mu$-a.e.\ $x$. Hence we have
\begin{equation}\label{eq.rep}
 \PP (A) = \int_E \PP_x(A) \, d\mu (x ) \quad \forall\, A \in \mathscr{B} \,,
\end{equation}
This shows that uniqueness of solutions of the local martingale problem for $(\cL, \delta_x)$ for all $x \in E$
implies uniqueness of the solution of the local martingale problem for $(\cL, \mu )$ for arbitrary initial 
distribution $\mu$.
\end{proof}

We end this section by establishing a result which allows us to construct solutions to equation 
$[A,F,G]$ from solutions of approximate equations $[A, F_n, G_n]$.

\begin{lem}\label{l.approx}
Suppose we are given sequences $(F_n)_{n\in\CN}$ and $(G_n)_{n\in\CN}$ which satisfy the assumptions 
of Hypothesis \ref{hyp1}, are continuous and are uniformly bounded on bounded sets. Furthermore, assume that $F_n(x)$ converges to $F(x)$ 
in $\tilde{E}$  and $G_n(x)$ converges to $G(x)$ in $\cL (H, \tilde{E})$, both convergences being uniform on the compact subsets of $E$.

If $\PP_n$ solves the martingale problem associated with equation $[A,F_n,G_n]$ and
if the sequence $(\PP_n)_{n\in\CN}$ is tight, then any accumulation point of the sequence solves the 
martingale problem associated with $[A,F,G]$.
\end{lem}

\begin{proof}
For a number $M \in \CR$ we put $\tau_M := \inf \{ t > 0 \, : \, \norm{\bx_t} \geq M\}$. Now fix 
$0 \leq s_1 < \cdots < s_N \leq s < t\, , \, N \in \CN$, and for $j =1, \ldots , N$ functions $h_j \in C_b(E)$ and 
$f = \varphi (\dual{\cdot }{x_1^*}, \ldots , \dual{\cdot}{x_m^*}) \in \mathscr{D}$.

We define $\Phi_n : C([0,\infty ); E) \to \CR$ by 
\[
 \Phi_n (\bx ) := \Big[ f(\bx_{t\wedge \tau_M}) - f(\bx_{s\wedge \tau_M}) - \int_s^t
\one_{[0,\tau_M]}(r)\big( L_nf\big)(\bx_r)\, dr\Big]\cdot \prod_{j=1}^N h_j(\bx_{s_j}),
\]
where $L_n:= L_{[A,F_n,G_n]}$.
Similarly, we define the function $\Phi$, replacing $L_n$ with $L:=L_{[A,F,G]}$.

Using the assumption that $F_n$ and $G_n$ 
are uniformly bounded on bounded subsets, it is easy to see that the sequence $\Phi_n$ is uniformly bounded.

The assumptions on the convergence of $F_n$ and $G_n$ imply that $L_nf$ converges to $Lf$, uniformly on the 
compact subsets of $E$. Now let a compact subset $\mathscr{C}$ of $C([0,\infty ); E)$ be given. By the Arzel\`a-Ascoli 
theorem, there exists a compact subset $K$ of $E$ such that $\bx_r \in K$ for all $0\leq r \leq t$, whenever $\bx 
\in \mathscr{C}$. Let $C := \prod_{j=1}^n \|h_k\|_\infty$. Given $\eps >0$, pick $n_0$ such that 
$|L_nf(x) - Lf(x)| \leq \eps$ for all $x \in K$, whenever $n\geq n_0$. Then, for $\bx \in \mathscr{C}$ and $n\geq n_0$
we have
\[
 |\Phi_n (\bx ) - \Phi (\bx )| \leq \int_s^t \one_{[0,\tau_M]}(r) |L_nf(\bx_r) - Lf(\bx_r)|\, dr \cdot C
\leq |t-s|\eps C,
\]
proving that $\Phi_n$ converges to $\Phi$ uniformly on compact subsets of $C([0,\infty ); E)$. 
\medskip 

Now let $\PP$ be an accumulation point of the sequence $(\PP_n)$. Passing to a subsequence, we may assume that
$\PP_n$ converges weakly to $\PP$. In particular, the sequence $(\PP_n)$ is tight. Thus, given $\eps>0$, we find a compact
set $\mathscr{C}$ of $C([0,\infty); E)$ such that $2c\PP_n(\mathscr{C}^c) \leq \eps$, where $c$ is such that 
$\|\Phi_n\|_\infty \leq c$. It follows that
\[ \Big| \int \Phi \, d\PP - \int \Phi_n \, d\PP_n \big|
 \leq \Big| \int \Phi \, d\PP - \int \Phi \, d\PP_n \big| + \eps +  \sup_{\bx \in \mathscr{C}}|\Phi (\bx) - \Phi_n(\bx)|.
\]
To conclude that $\int \Phi \, d\PP = \lim_{n\to\infty } \int \Phi_n\, d\PP_n = 0$, it remains to prove that $\int \Phi \, d\PP_n$ converges to $\int \Phi \, d\PP$.
We know that $\PP_n$ converges weakly to $\PP$. Unfortunately, the function $\Phi$ is not continuous. However, it is continuous at all points 
$\by$ at which the map $\bx \mapsto \tau_M(\bx)$ is continuous. Moreover, it can be proved that the set of all $M$ such that
$\PP (\{\by : \tau_M\,\, \mbox{is discontinous at}\, \by\} ) > 0$ is countable, see \cite[Lemma 3.5 and 3.6]{hs12} (see also Sections VI.2 and VI.3 of
\cite{js}). We can thus find a number $M$ such that $\Phi$ is continuous except for a $\PP$-null set. As is well known, see \cite[Cor. 8.4.2]{bogachev},
this together with the weak convergence of the $\PP_n$ suffices to conclude that $\int \Phi d\PP_n \to \int \Phi\, d\PP$, as desired and
it follows that $\int \Phi\, d\PP =0$.

 Since the sampling points 
$(s_j)$ and $s, t$ as well as the functions $h_j$ were arbitrary, it follows from a monotone class argument that 
\[ f(\bx_{t\wedge \tau_M}) - f(\bx_{0\wedge\tau_M})-\int_0^t \one_{[0,\tau_M]}(r)Lf(\bx_r)\, dr \]
is a martingale under $\PP$. Since $f$ was arbitrary, and we can pick a sequence $M_k \uparrow \infty$ such that the above is true, we have 
proved that 
$\PP$ solves the local martingale problem associated with 
equation $[A,F,G]$.
\end{proof}

As a corollary, we obtain a sufficient condition for the Feller property of the associated transition semigroup.

\begin{cor}
Assume that equation $[A,F,G]$ is well-posed and that $F$ and $G$ are continuous. We denote by $\mathscr{T}$ the transition semigroup 
for the associated martingale problem for $\cL_{[A,F,G]}$ and by $\PP_\mu$ the unique solution 
of the local martingale problem for $(\cL_{[A,F,G]}, \mu)$. The following are equivalent
\begin{enumerate}
 \item The map $\mu \mapsto \PP_\mu$ is continuous from $\cP (E)$ to $\cP (C([0,\infty); E))$ 
where both are endowed with their respective weak topology.
 \item If $x_n \to x$ in $E$, then the set $\{ \PP_{x_n}\,:\, n \in \CN\}$ is tight. 
\end{enumerate}
In this case, the semigroup $\mathscr{T}$ has the Feller property, i.e.\ $\mathscr{T}(t)f \in C_b(E)$ for all 
$f \in C_b(E)$.
\end{cor}

\begin{proof}
(1) $\Rightarrow$ (2): If $x_n \to x$ then $\delta_{x_n} \to \delta_x$ weakly. In particular, 
$\{\delta_{x_n}\,:\, n \in \CN\}$ is relatively weakly compact. By (1) the set $\{\PP_{x_n}\,:\,  
n\in\CN\}$ is relatively weakly compact hence tight.\medskip 

(2) $\Rightarrow$ (1): Let $x_n \to x$. By (2), $\{\PP_{x_n}\,:\, n \in \CN\}$ is tight. By Lemma 
\ref{l.approx} any accumulation point of the $\PP_{x_n}$ must solve the local martingale problem 
for $\cL_{A,F,G}$. Since every accumulation point also must have initial distribution $\delta_x$, well-posedness 
implies that the only accumulation point is $\PP_x$. Now a subsequence-subsequence argument yields that 
$\PP_{x_n}$ converges weakly to $\PP_x$. This proves that the map $x \mapsto \PP_x$ is continuous from $E$ to 
$\cP (C([0,\infty); E))$.

It follows from the proof of uniqueness in Theorem \ref{t.markov}, namely from equation \eqref{eq.rep}, that 
\[
 \int \Phi \, d\PP_\mu = \int_E \int \Phi \, d\PP_x \, d\mu (x),
\]
for all bounded, continuous functions $\Phi$ on $C([0,\infty); E)$. With this representation the
continuity of $\mu \mapsto \PP_\mu$ follows.\medskip 

If (1) or, equivalently, (2) is satisfied, then the Feller property of $\mathscr{T}$ follows from the identity 
$\mathscr{T}(t)f(x) = \int f\circ \pi_t \, d\PP_x$ and the fact that $f\circ \pi_t$ is a bounded, continuous function 
on $C([0,\infty); E)$.
\end{proof}

\section{Yamada-Watanabe theory}\label{sect.yw}

In view of Theorem \ref{t.weakmart}, the uniqueness requirement for the local martingale problem associated 
with \eqref{eq.sde} is equivalent with the requirement that whenever $\bX_1$ and $\bX_2$ are weak solutions
of \eqref{eq.sde}, possibly defined on different probability spaces, 
such that $X_1(0)$ and $X_2(0)$ have the same distribution $\mu$, then $\bX_1$ and $\bX_2$ have the same distribution 
as $C([0,\infty ); E)$-valued random variables. In this situation, one says that \emph{uniqueness in law} or \emph{uniqueness in 
distribution} holds. 

In some cases, in particular in the case of Lipschitz continuous coefficients, it is easier to verify a different
notion of uniqueness.

\begin{defn}
We say that \emph{pathwise uniqueness} holds for solutions of equation \eqref{eq.sde}
 if whenever  $((\Omega, \Sigma,\FF,  \P), W_H, \bX_j)$ are weak solution of \eqref{eq.sde} for $j=1,2$ 
with $X_1(0) = X_2(0)$ almost surely, then $\P (X_1(t) = X_2(t)\, \forall\, t \geq 0)=1$. 
\end{defn}

A classical result of Yamada and Watanabe \cite{yw1} asserts that in the case where $E = \CR^d$ and $W_H$ is a finite 
dimensional Brownian motion, i.e.\ $H$ is finite-dimensional, pathwise uniqueness implies uniqueness in law. 
Pathwise uniqueness also has other far-reaching consequences, most notably, it implies the \emph{strong} existence of 
solutions.

\begin{defn}
A weak solution $((\Omega, \Sigma, \FF, \P), W_H, \bX)$ is said to \emph{exist strongly} if  $\bX$ is adapted 
to the filtration $\mathds{G} := (\cG_t)_{t \geq 0}$, where $\cG_t$ is the augmentation of 
$\sigma (X(0), W_Hh_k(s) \,:\, s \leq t, k \in I)$. Here, $(h_k)_{k \in I}$ is a finite or countably infinite orthonormal basis of $H$.
\end{defn}

A priori, strong existence of solutions is a mere measurability requirement. This requirement captures the idea that the information needed to 
construct a solution to a stochastic differential equation is already contained in the initial datum and the Wiener process. Of particular 
importance in applications is the fact that given pathwise uniqueness solutions can be constructed on a \emph{given} stochastic basis and 
with respect to a \emph{given} $H$-cylindrical Wiener process, see Corollary \ref{c.strongexistence}.

Ondrej\'at \cite{onddiss} has generalized the Yamada-Watanabe results to the situation where $E$ is a 2-smoothable Banach space.
One of the main difficulties he had to overcome was to prove that  distributional copies of solutions are again solutions. 
As he was working with the concept of mild solutions, this required a 
detailed study of the distributions of Banach space valued stochastic integrals. In our situation, with the concept of weak solutions, 
the proof is easier and can in fact be reduced to the finite dimensional situation.

\begin{thm}\label{t.yw}
Pathwise uniqueness for \eqref{eq.sde} implies uniqueness in law. Moreover, every solution of \eqref{eq.sde} exists strongly.
\end{thm}

For the convenience of the reader, we include a full proof which follows closely the proof in the finite dimensional situation. It is also possible to 
show that our situation fits into the abstract framework considered in \cite{kurtz07} and to obtain Theorem \ref{t.yw} from the results proved there.

\begin{proof}
Let two weak solutions $((\Omega_j, \Sigma_j, \FF_j, \P_j), W_H^j, \bX_j)$ of equation \eqref{eq.sde} be given such that 
$X_1(0)$ and $X_2(0)$ have the same distribution $\mu$. We first define distributional copies of these two solutions on a common 
stochastic basis.

To that end, we fix an orthonormal Basis $(h_n)_{n\in \CN}$ (the case where $H$ is finite dimensional is similar) of $H$ and define 
the measure $\PP_j$ on the Borel $\sigma$-algebra of 
\[
 \tilde{\Omega} := C([0,\infty); E) \times E \times C([0,\infty); \CR^\infty), 
\]
viewed as the countable product of Polish spaces,
as the image of $\P_j$ under the map
\[
\omega_j \mapsto \big(X_j (\cdot, \omega_j) - X_j (0, \omega_j), X_j(0,\omega_j), (H_H^j(\cdot, \omega_j)h_n)_{n \in \CN}\big) \,
\]
A typical element of $\tilde{\Omega}$ will be denoted by $(\by, x_0, \bw)$. Note that the projection of $\PP_j$  to $C([0,\infty); \CR^\infty)$
is the countable product of Wiener measure; we denote this measure by $\mathds{W}$. Thus, under $\PP_j$, 
the random element $(x_0, \bw)$ has distribution $\mu\otimes \mathds{W}$. 

We let $\QQ_j$ be a regular conditional distribution of $\by$ given 
$(x_0, \bw)$ under $\PP_j$, i.e.\ $\QQ_j (x_0, \bw, \cdot)$ is a probability measure on $\cB(C([0, \infty); E))$ for all $x_0 \in E$ and $\bw \in C([0,\infty); \CR^\infty)$ 
and given sets $A \in \cB(C([0, \infty); E))$, $B \in \cB(E)$ and  $C \in \cB(C([0, \infty); \CR^\infty))$, we have 
\[
\PP_j(A\times B\times C) = \int_{B\times C} \QQ_j (x_0, \bw, A) \, d(\mu\otimes \mathds{W})(x_0, \bw) .
\]
We now define distributional copies of the solutions on a common probability space. We put 
\[
 \Omega := C([0,\infty); E) \times C([0, \infty); E) \times E \times C([0,\infty); \CR^\infty), 
\]
and denote a canonical element of $\Omega$ by $(\by_1, \by_2, x_0, \bw)$. We define the measure $\PP$ on the Borel
$\sigma$-algebra $\Sigma$ of $\Omega$ by 
\[
\PP (A\times B \times C \times D) := \int_{C\times D} \QQ_1(x_0, \bw, A) \QQ_2(x_0, \bw, B) \, d(\mu\otimes \mathds{W})(x_0, \bw).
\]
Finally, we define $\cG_t := \sigma (x_0, \by_1(s), \by_2(s), \bw(s) : s \leq t)$, $\cF_t$ as the augmentation of $\cG_{t+}$ by the $\PP$-null sets and set $\FF := (\cF_t)_{t \geq 0}$. 
As in the finite dimensional case, see \cite[Lemma IV.1.2]{iw}, we see that for every $k \in \CN$ the $k$-th component $\bw_k$ of $\bw$ is a Brownian motion with 
respect to $\FF$. 

As $\bw_k$ and $\bw_l$ are independent for $k\neq l$, we can define an $H$-cylindrical Wiener process with respect
to $\FF$ by setting, for $f \in L^2(0, \infty; H)$
\[
W_H(f) := \sum_{k=1}^\infty \int_0^\infty [f(t),h_k]_H\, d\bw_k (t) .
\]

We claim that $((\Omega, \Sigma, \FF, \PP), W_H, x_0 + \by_j)$ is a weak solution of equation \eqref{eq.sde} for $j=1,2$. We will write 
$\bx_j := x_0 + \by_j$ for $j=1,2$.
To prove the claim, let
$x^* \in D(A^*)$ be fixed. Using the measurability of $F$ and $G$, as well as the continuity of the functionals $x^*$ resp.\ $A^*x^*$, 
it follows from the definitions above that the joint distribution of 
\[
\big( \dual{X_j(0)}{x^*}, \dual{F(X_j(\cdot ))}{x^*}, \dual{X_j(\cdot)}{A^*x^*}, ([G(X_j(\cdot))^*x^*, h_k])_{k \in \CN}, (W_H^j(\cdot)h_k)_{k \in \CN}\big)
\]
under $\P_j$ is the same as that of 
\[
\big( \dual{x_0}{x^*}, \dual{F(\bx_j(\cdot))}{x^*}, \dual{\bx_j(\cdot) }{A^*x^*}, ([G(\bx_j (\cdot))^*x^*, h_k])_{k \in \CN}, (W_H(\cdot)h_k)_{k \in \CN}\big)
\]
under $\PP$. Thus, for fixed $t \geq 0$, we infer as in the finite dimensional situation that for $j=1,2$ and every $n \in \CN$ the distribution of
\[
\begin{aligned}
Z_{j,n}(t) := X_j (t) - & \dual{X_j (0)}{x^*} - \int_0^t \dual{X_j(s)}{A^*x^*}\, ds - \int_0^t \dual{F(X_j(s))}{x^*}\, ds\\
&  - \sum_{k=1}^n 
\int_0^t [G(X_j(s))^*x^*, h_k]\, dW_H^j(s)h_k
\end{aligned}
\]
under $\P_j$ is the same as that of 
\[
\begin{aligned}
\mathbf{z}_{j,n}(t) ;=\bx_j (t) - & \dual{\bx_j (0)}{x^*} - \int_0^t \dual{\bx_j(s)}{A^*x^*}\, ds - \int_0^t \dual{F(\bx_j(s))}{x^*}\, ds\\
&  - \sum_{k=1}^n 
\int_0^t [G(\bx_j(s))^*x^*, h_k]\, dW_H^j(s)h_k
\end{aligned}
\]
under $\PP$. Since $\bX_j$ is a solution of equation \eqref{eq.sde}, $Z_{j,n}(t) \to 0$ $\P_j$-almost surely as $n\to \infty$, hence $\mathbf{z}_{j,n}$ 
converges to $0$ in distribution and thus $\PP$-almost surely. Since $t \geq 0$ and $x^* \in D(A^*)$ were arbitrary, this proves that $\bx_j$ is indeed 
a weak solution.

As $x_0 + \by_1$ and $x_0 + \by_2$ are weak solutions defined on the same stochastic basis and with respect to the same 
$H$-cylindrical Wiener process, pathwise uniqueness implies that $x_0 + \by_1 = x_0 + \by_2$ $\PP$-almost surely. This, in turn, implies that 
the random elements $\bX_j$ have the same distribution.\medskip

As for the strong existence of solutions, define for $x_0 \in E$ and $\bw \in C([0, \infty); \CR^\infty)$ the measure 
$\mathbf{R}(x_0, \bw, \cdot)$ on the Borel $\sigma$-algebra $\mathscr{S}$ of $C([0, \infty); E) \times C([0, \infty); E)$ as the product of 
$\QQ_1(x_0, \bw, \cdot)$ and $\QQ_2(x_0, \bw, \cdot)$. Then, for $G \in \mathscr{S}$, $C \in \cB (E)$ and $D \in \cB (C([0, , \infty); \CR^\infty)$ we have
\[
\PP (G\times C \times D) = \int_{C\times D} \mathbf{R}(x_0, \bw, G)\, d(\mu\otimes \mathds{W})(x_0, \bw).
\]
Now consider $\Lambda := \{(\by_1, \by_2)\,:\, \by_1 = \by_2\}$. It follows from the  first part of the proof that
$R(x_0, \bw, \Lambda) =1$ for $(\mu\otimes\mathds{W})$-almost every $(x_0, \bw)$, say outside the set $N \in \cB (E)\otimes \cB (C([0, \infty); 
E))$ with $(\mu\otimes \mathds{W})(N) = 0$. Using Fubini's theorem, we find for $(x_0, \bw) \in N^c$
\[
1= \mathbf{R}(x_0, \bw, \Lambda) = \int_{C([0, \infty); E)} \QQ_1(x_0, \bw, \{\by\})\, \QQ_2(x_0, \bw, d\by).
\]
As all measures involved in this equation are probability measures, this can only happen if $\QQ_1(x_0, \bw, \{\by_0\})= \QQ_2(x_0, \bw, \{\by_0\})
=1$ for a certain $\by_0 = \Phi (x_0, \bw) \in C([0, \infty); E)$. 

A straightforward generalization of the proof in the finite-dimensional case, see \cite[Section 5.3.D]{ks}, shows that the map $\Phi : E \times C([0, \infty); \CR^\infty) \to C([0, \infty); E)$ is $\cB (E) \otimes \cB ( C([0, \infty); \CR^\infty))/ \cB ([0, \infty); E)$-measurable.
Moreover, if we define $\mathscr{H}_t$ as the augmentation of $\cB(E) \otimes \sigma (\bw (s): s\leq t)$ by the $\mu\otimes \mathds{W}$-null sets
and $\mathscr{I}_t := \sigma (\by (s): s \leq t)$, then $\Phi$ is $\mathscr{H}_t/\mathscr{I}_t$-measurable for every $t>0$. 

By what was done so far, $x_0 + \by_j = x_0 + \Phi (x_0, \bw)$ $\PP$-almost surly. Thus, for $j=1,2$, we have $\bX_j = X_j(0) + \Phi (X_j(0), 
(W_H^j(\cdot)h_n)_{n \in \CN})$ $\P_j$-almost surely. The measurability properties of $\Phi$ now imply that the solution 
$((\Omega_j, \Sigma_j, \FF, \PP_j),  W_H^j, \bX_j)$ exists strongly for $j=1,2$.
\end{proof}

As a consequence of pathwise uniqueness, we find solutions of equation \eqref{eq.sde} on a \emph{given} probability space 
and with respect to a \emph{given} $H$-cylindrical Wiener process.

\begin{cor}\label{c.strongexistence}
Assume that pathwise uniqueness holds for equation $[A, F, G]$ and that for some $\mu \in \cP (E)$, there exists a weak solution 
of $[A,F,G]$ with initial distribution $\mu$. Then, given any stochastic basis $(\Omega, \Sigma, \FF, \PP)$ on which an $H$-cylindrical 
Wiener process $W_H$ with respect to $\FF$ is defined and on which an $\cF_0$-measurable random variable $\xi$ with distribution 
$\mu$ is defined, there exists a process $\bX$ such that $((\Omega, \Sigma, \PP), \FF, W_H, \bX)$ is a weak solution of equation $[A, F, G]$
with $X(0) = \xi$.
\end{cor}

\begin{proof}
Let $((\Omega', \Sigma', \PP'), \FF', W_H', \bX')$ be a weak solution of $[A,F,G]$ with $X'(0) \sim \mu$. The proof of Theorem \ref{t.yw}
yields that $X' = X(0) + \Phi (X'(0), (W_H'(\cdot)h_n)_{n\in\CN})$. We put $X := \xi + \Phi (\xi, (W_H(\cdot)h_n)_{n \in \CN})$.

Then the distribution of $(X'(0), \bX', (W_H'(\cdot)h_n)_{n\in\CN})$ under $\PP'$ is the same as the distribution of 
$(\xi, \bX, (W_H(\cdot)h_n)_{n\in\CN})$ under $\PP$. Arguing as in the first part of the proof of Theorem \ref{t.yw}, it follows 
that $((\Omega, \Sigma, \PP), \FF, W_H, \bX)$ is a weak solution of equation $[A, F, G]$
with $X(0) = \xi$.
\end{proof}

\section{Stochastic integration and mild solutions}\label{sect.integration}

We now address the question whether weak solutions of \eqref{eq.sde} are also mild solutions, i.e.\ 
for all $t \geq 0$ the $\cL (H, E)$-valued process $s \mapsto S_{t-s}G(X_s)$ is stochastically integrable 
(in a sense to be made precise below) and we have,
almost surely, 
\begin{equation}\label{eq.evalued}
X_t = X_0 + \int_0^tS_{t-s}F(X_s)\, ds + \int_0^t S_{t-s}G(X_s)\, dW_H(s) \,. 
\end{equation}

Having mild solutions, rather than weak solutions, has many advantages. In particular, one can make use of the 
\emph{factorization method} \cite{dpkz}. The factorization method is useful to prove continuity of the paths of 
solutions which we have assumed throughout and also to establish the tightness assumption in Lemma \ref{l.approx}, 
thus enabling us to construct solutions to stochastic differential equations. 

In this section, we prove the equivalence of weak and mild solutions under additional assumptions 
on either equation $[A,F,G]$ or the state space $E$. As an intermediate step, we first consider \emph{weakly mild solutions}
in which we only require \eqref{eq.evalued} to hold when tested against functionals $x^* \in \tilde{E}^*$.

\subsection{Weakly mild solutions}\label{ssect.wms}

\begin{defn}\label{def.mild}
A tuple $\big( (\Omega, \Sigma, \FF, \P), W_H, \bX\big)$, where $(\Omega, \Sigma , \FF, \P)$ is stochastic basis satisfying the usual conditions, 
$W_H$ is an $H$-cylindrical Wiener process with respect to $\FF$ and
$\bX$ is a continuous, $\FF$-progressive, $E$-valued process is called a \emph{weakly mild solution} of \eqref{eq.sde} 
if for all $x^* \in \tilde{E}^*$ and $t \geq 0$ we have
\begin{equation}\label{eq.mildsolution}
\dual{X_t}{x^*}  =  \dual{S_tX_0}{x^*} + \int_0^t \dual{S_{t-s}F(X_s)}{x^*} \, ds
 +  \int_0^t G(X_s)^*S_{t-s}^*x^* dW_H(s) .
\end{equation}
$\P$-a.e.
\end{defn}

\begin{rem}\label{l.welldefined}
By our assumptions on the coefficients $A,F$ and $G$, the
Lebesgue-integral and the stochastic integral in \eqref{eq.mildsolution} are well-defined
for all $t\geq 0$ and $x^* \in E^*$.

Indeed, the map $(s,\omega ) \mapsto F(X(s,\omega ))$ is measurable as a composition of two measurable maps. 
Hence, it is the limit of a sequence of simple functions $f_n$ almost everywhere with respect to 
$ds\otimes \P$. Thus
\[ \dual{S(t-\cdot )F(X)}{x^*} = \lim \dual{f_n}{S(t-\cdot )^*x^*} \quad ds\otimes \P-a.e. \]
We have  $\dual{f_n}{S(t-\cdot )^*x^*} = \sum_{j=1}^{N_n} \one_{A_{jn}}\dual{x_{jn}}{S(t-\cdot )^*x^*}$
for certain measurable sets $A_{jn}$ and vectors $x_{jn} \in \tilde{E}$ and this is measurable since 
$s\mapsto \dual{x}{S(t-s)^*x^*}$ is continuous for all $x \in \tilde{E}$ and $x^* \in \tilde{E}^*$.
Hence $\dual{S(t-\cdot )F(X)}{x^*}$ is the limit of measurable functions $ds\otimes \P$ almost everywhere 
and thus measurable. In view of the continuity of the paths of $\bX$, the boundedness of $F$ on bounded sets
and the boundedness of $S$ on finite time intervals, it follows that for almost all $\omega$ the 
function $s \mapsto \dual{S(t-s)F(X(s,\omega ))}{x^*}$ is bounded, hence integrable. 

The stochastic integral can be dealt with similarly, using the series expansion
\[
  G(X(s,\omega))^*S(t-s)^*x^* = \sum_k \dual{G(X(s,\omega ))h_k}{ S(t-s)^*x^*}_H h_k
\]
where $(h_k)$ is a finite or countably infinite orthonormal basis of $H$. 
\end{rem}

We now prove that the notions `weak solution' and `weakly mild solution' are equivalent. 
Under additional assumptions which ensure that the stochastic convolution is well-defined,
variations of this result (for mild solutions) have been proved in various settings, 
see \cite[Theorem 5.4]{dpz}, \cite[Theorem 7.1]{vNW05} or \cite[Proposition 3.3]{sv10}. Assuming that $G$ is constant or 
that $E$ is a UMD Banach space, in the following 
subsection we prove that weakly mild solutions are mild solutions. In particular, it \emph{follows} that the stochastic 
convolution is well-defined.
\medskip

We note that the adjoint semigroup $S^*$ may not be strongly continuous, which causes technical difficulties.
To overcome these, we will use results about the \emph{$\odot$-dual semigroup} $S^\odot$. We recall some 
basic definitions and properties and refer the reader to \cite{vNadjoint} for more information.

Define $\tilde{E}^\odot := \overline{D(A^*)}$. Then $\tilde{E}^\odot$ is a closed, \ws-dense subspace
of $\tilde{E}^*$ which is invariant under the adjoint semigroup. The restriction of the adjoint semigroup to 
$\tilde{E}^\odot$, denoted
by $S^\odot$, is strongly continuous. In fact, $\tilde{E}^\odot = \{ x^* \in \tilde{E}^*\, : \, t \mapsto S(t)^*x^*\,\,
\mbox{is strongly continuous} \}$. We denote by $A^\odot$ the generator of $S^\odot$. Note that $A^\odot$ is
exactly the part of $A^*$ in $\tilde{E}^\odot$.

\begin{prop}\label{p.equivalence}
The weak and the weakly mild solutions of \eqref{eq.sde} coincide. 
\end{prop}

\begin{proof}
First assume that $\bX$ is a weak solution. For $n \in \CN$, define
 \[\tau_n :=  \inf \{ t > 0 \, : \norm{X(t)} \geq n \}\,.\] 
Since $\bX$ is a weak solution, we have for $x^* \in D(A^*)$ and $t\geq 0$
\[ 
\begin{aligned}
 \langle X_{t\wedge \tau_n} &, x^*\rangle =  \dual{X_{0\wedge \tau_n}}{x^*} +
\int_0^t \one_{[0,\tau_n]}(s)\dual{X_s}{A^*x^*}\, ds\\
& + \int_0^t\one_{[0,\tau_n]}(s)\dual{F(X_s)}{x^*}\, ds
+ \int_0^t \one_{[0,\tau_n]}(s)G(X_s)^*x^*\, dW_H(s)
\end{aligned}
\]
almost surely. In view of Remark \ref{rem1}, we may (and shall) assume that the exceptional set does not
depend on $t$. Below, we will suppress the statement $\P$-almost surely.

Fix $t >0$ and let $f \in C^1([0,t])$ and $x^* \in D(A^*)$. Putting $\varphi := f \otimes x^*$, 
It\^o's formula yields
\begin{equation}\label{eq.var}
\begin{aligned}
  \langle X_{t\wedge \tau_n },\varphi (t)\rangle & = \dual{X_{0\wedge \tau_n }}{\varphi (0)} + \int_0^{t} 
 \dual{X_{s\wedge \tau_n}}{\varphi'(s)}\, ds
 + \int_0^{t\wedge \tau_n} \dual{X_s}{A^*\varphi (s)}\, ds\\ 
 & \quad\quad+ \int_0^{t\wedge \tau_n} \dual{F(X_s)}{\varphi (s)}\, ds + \int_0^t 
\one_{[0,\tau_n]}(s) G(X_s)^*\varphi (s)\, dW_H(s)\,.
\end{aligned}
\end{equation}
By linearity, the above equation holds for $\varphi = \sum_{k=1}^N f_k\otimes x_k^*$ where 
$f_k \in C^1([0,t])$ and $x_k^* \in D(A^*)$. Since $D(A^\odot )$ is a Banach space with 
respect to the graph norm, so is $C^1([0,t];D(A^\odot ))$. 
Functions of the form $\varphi := \sum_{k=1}^n f_k \otimes x_k^*$
with $f_k \in C^1([0,t])$ and $x_k^*\in D(A^\odot)$ for $1\leq k \leq n$ are 
dense in $C^1([0,t]; D(A^\odot ))$ and hence an approximation argument shows that \eqref{eq.var} holds
for all $\varphi \in C^1([0,t]; D(A^\odot) )$.\medskip 

Now let $x^* \in D((A^\odot)^2)$ and $\varphi (s) = S_{t-s}^* x^*$. 
Then $\varphi \in C^1([0,t];D(A^\odot))$ with
$\varphi' (s) = - S_{t-s}^* A^*x^*$. Let us note that $\int_0^t \dual{X_{s\wedge \tau_n}}{\varphi'(s)}\, ds
= \int_0^{t\wedge \tau_n} \dual{X_s}{\varphi'(s)}\, ds + \int_{t\wedge \tau_n}^t \dual{X_{\tau_n}}{\varphi'(s)}\, ds$, where  the last term is zero if $\tau_n \geq t$. Thus equation
\eqref{eq.var} yields for this $\varphi$
\begin{equation}\label{eq.solmild}
\begin{aligned}
 & \dual{X_{t\wedge \tau_n }}{x^*} =  \dual{S_tX_{0\wedge \tau_n }}{x^*} + \int_0^{t\wedge\tau_n} 
\dual{S_{t-s}F(X_s)}{x^*}\, ds\\
& \quad \quad  - \int_{t\wedge \tau_n}^t \dual{S_{t-s}X_{\tau_n}}{A^*x^*}\, ds + \int_0^t \one_{[0,\tau_n]} 
G(X_s)^*S_{t-s}^* x^*\, dW_H(s)\,.
\end{aligned}
\end{equation}

We next want to extend \eqref{eq.solmild} to arbitrary $x^* \in \tilde{E}^*$. Obviously, the term $\int_{t\wedge \tau_n}^t \dual{S_{t-s}X_{\tau_n}}{A^*x^*}$
is not well-defined for arbitrary $x^* \in \tilde{E}^*$. However, using the well-known fact that for $0 \leq a < b$ and $x \in \tilde{E}$ the 
integral $\int_a^b S(s)x\, ds$ belongs to the domain  of the generator $A$ and $A\int_a^b S(s)x\, ds = S(b)x - S(a)x$, it follows that 
\[
\int_{t\wedge \tau_n}^t \dual{S_{t-s}X_{\tau_n}}{A^*x^*}\, ds = \dual{ S_{t-t\wedge \tau_n}X_{\tau_n} - X_{\tau_n}}{x^*}.
\]
Since $D((A^\odot )^2)$ is sequentially
\ws-dense in $\tilde{E}^*$, given $z^* \in \tilde{E}^*$, we find a sequence $x_k^* \in D((A^\odot)^2)$ such that
$x_k^* \weak^* z^*$. Arguing similar as in the proof of Lemma \ref{l.count}, we find a sequence
$y_m^*$ in the convex hull of the $(x_k^*)$ such that $y_m^* \weak^* z^*$ and 
\[ \one_{[0,\tau_n]}(\cdot ) G(X(\cdot ))^*S(t-\cdot )^*y_m^* \to \one_{[0,\tau_n]}
(\cdot )G(X(\cdot ))^*S(t-\cdot )^*z^* \]
in $L^2(\Omega)$. Thus, since $\expect \big|\int_0^t \Phi (s)\, dW_H(s)\big|^2 = 
\|\Phi\|_{L^2(\Omega;L^2([0,t];H))}^2$ we see that
\[ \int_0^t \one_{[0,\tau_n]}G(X(s ))^*S(t-s )^*y_m^*\, dW_H(s) \to 
 \int_0^t \one_{[0,\tau_n]}G(X(s ))^*S(t-s )^*z^*\, dW_H(s)
\]
in $L^2(\Omega ; L^2(0,t;H))$. Passing to a subsequence, we may assume that we have convergence almost everywhere. 
Moreover, since \eqref{eq.solmild} also holds for $x^*=y_m^*$, for all 
$m \in \CN$, noting that
\[
\one_{[0,\tau_n]}(s)\big| \dual{S(t-s)F(X(s))}{y_m^*} \big|\leq \one_{[0,\tau_n]}(s)Me^{\omega (t-s)} B_n
\cdot \sup_{m\in\CN} \norm{y_m^*},
\]
where $M$ and $\omega$ are such that $\|S(t)\|\leq Me^{\omega t}$ for $t \geq 0$ and 
$B_n :=\sup\{\norm{F (x)} \, : \norm{x} \leq n\}$, is follows from
dominated convergence that $\int_0^{t\wedge \tau_n} \dual{S_{t-s}F(X_s)}{y_m^*}\, ds$ converges to 
$\int_0^{t\wedge \tau_n} \dual{S_{t-s}F(X_s)}{z^*}\, ds$ almost surely. It altogether we see that
\begin{equation}\label{eq.stoppedmildsol}
\begin{aligned}
 & \dual{X_{t\wedge \tau_n }}{z^*} =  \dual{S_tX_{0\wedge \tau_n }}{z^*} + \int_0^{t\wedge\tau_n} 
\dual{S_{t-s}F(X_s)}{z^*}\, ds\\
& \quad + \dual{ X_{\tau_n}-S_{t-t\wedge \tau_n}X_{\tau_n} }{z^*} + \int_0^t \one_{[0,\tau_n]} 
G(X_s)^*S_{t-s}^* z^*\, dW_H(s)\,.
\end{aligned}
\end{equation}
Upon letting $n \to \infty$, \eqref{eq.mildsolution} is proved for arbitrary $x^* = z^*$.\medskip

We now prove the converse and assume that $\bX$ is a weakly mild solution of \eqref{eq.sol}.
Fix $x^* \in D(A^*)$ and $t >0$. Then for $0<s<t$ we have
\begin{equation}\label{eq.mildsol}
\begin{aligned}
 \dual{X_s}{A^*x^*} & = \dual{S_sX_0}{A^*x^*} +  \int_0^s \dual{S_{s-r}F(X_r)}{A^*x^*}\, dr\\
& \quad + \int_0^s G(X_r)^*S_{s-r}^*A^*x^*\, dW_H(r)
\end{aligned}
\end{equation}
almost surely. We note that the exceptional set may depend $s$. However, all terms in this equation are 
jointly measurable in $s$ and $\omega$. Hence, the left-hand side and the right-hand side of \eqref{eq.mildsol}
are equal as elements of $L^0((0,t); L^0(\Omega ))$. By the canonical isomorphism $L^0((0,t); L^0(\Omega ))
\simeq L^0(\Omega; L^0(0,t))$, there exists a set $N \subset \Omega$ with $\P(N) = 0$ such that 
outside $N$ equation \eqref{eq.mildsol} holds as an equation in $L^0(0,t)$, i.e.\ for almost every $s \in (0,t)$,
where the exceptional set may depend on $\omega$.  Next note that by the continuity of the paths, the local boundedness
of $S$ and the boundedness of $F$ on bounded sets, the first three terms are, as functions of $s$, $\P$-almost surely 
bounded on $(0,t)$ and hence belong to $L^1(0,t)$. Possibly enlarging $N$, we may 
assume that outside $N$ equation \eqref{eq.mildsol} holds as an equation in $L^1(0,t)$.
Integrating from $0$ to $t$, we find that, $\P$-almost surely, we have
\begin{equation}\label{eq.integratedsol}
\begin{aligned}
  \int_0^t \dual{X_s}{A^*x^*}\, ds =  &
\int_0^t \dual{S_sX_0}{A^*x^*}\, ds
 + \int_0^t\int_0^s\dual{S_{s-r}F(X_r)}{A^*x^*}
\, dr\, ds\\
& \quad + \int_0^t \int_0^s G(X_r)^*x^*S_{s-r}^*A^*x^*\, dW_H(r)\, ds \,.
\end{aligned}
\end{equation}
Recall that for $x^* \in D(A^*)$ we have $\int_0^t S(s)^*A^*x^*\, ds = S(t)^*x^* - x^*$ for all $t\geq 0$. Here,
the integral has to be understood as \ws-integral. Using this, we obtain, pathwise, 
\[  \int_0^t\dual{S_sX_0}{A^*x^*}\, ds= \Big\langle X_0, \int_0^tS_s^*A^*x^*\, ds\Big\rangle
 =  \dual{X_0}{S_t^*x^* - x^*}
= \dual{S_tX_0-X_0}{x^*}.
\]
Using Fubini's theorem, we have  
\[ \begin{aligned}
\int_0^t\int_0^s &\dual{S_{s-r}F(X_r)}{A^*x^*}\, dr\,ds = 
\int_0^t\Big\langle F(X_r), \int_r^t S_{s-r}^*A^*x^*\Big\rangle\, ds\, dr\\
& =  \int_0^t\dual{S_{t-r}F(X_r)}{x^*}\, dr - \int_0^t \dual{F(X_r)}{x^*}\, dr 
\end{aligned}
\]
pathwise. Using the stochastic Fubini theorem \cite[Theorem 3.5]{vNV06}, it follows that
\[ 
 \begin{aligned}  
\int_0^t\int_0^s  & G(X_r)^*S_{s-r}^*A^*x^*\, dW_H(r)\, ds
= \int_0^t\int_r^t G(X_r)^*S_{s-r}^*A^*x^*\, ds \, dW_H(r)\\
& =   \int_0^t G(X_r)^*S_{t-r}^*x^* \, dW_H(r) - \int_0^t G(X_r)^*x^*\, dW_H(r) 
 \end{aligned}
\]
$\P$-almost surely.\smallskip
 
Plugging these three identities into \eqref{eq.integratedsol} and using that $\bX$ is a 
mild solution, \eqref{eq.sol} follows.
\end{proof}

Since all terms appearing in \eqref{eq.sol} are almost surely continuous, there is no problem 
in writing an equation for the stopped process $\dual{X_{t\wedge \tau}}{x^*}$ and we did this in 
the proof of Proposition \ref{p.equivalence}. On the other hand, for weakly mild solutions, the 
integrand in the stochastic integral changes with $t$, causing problems to obtain an equation 
for the stopped process. In \cite[Appendix]{bms05}, this problem was solved under the assumption 
that the stochastic convolution is almost surely continuous. In the proof of Proposition 
\ref{p.equivalence}, we have shown that for a weak solution, \eqref{eq.stoppedmildsol} holds for all $x^* \in \tilde{E}^*$. 
Given a stopping time $\tau$, we can repeat the arguments with $\tau_n$ replaced with $\tau_n\wedge \tau$ to obtain

\begin{cor}\label{c.stopped}
If $\bX$ is a weak (equivalently, weakly mild) solution of \eqref{eq.sde} and $\tau$ is a stopping time, 
then for all $t \geq 0$ and $x^* \in \tilde{E}^*$ we have
\begin{equation}\label{eq.stopped}
 \begin{aligned}
 \dual{X_{t\wedge \tau }}{x^*} &= \dual{S_t X_{0\wedge \tau }}{x^*} + \int_0^{t\wedge \tau} \dual{S_{t-s}F(X_s)}{x^*}\, ds\\
&   + \dual{X_\tau - S_{t-t\wedge \tau}X_\tau}{x^*}\one_{\{ \tau < \infty\}} + \int_0^t \one_{[0,\tau]}(s)
G(X_s)^*S_{t-s}^* x^*\, dW_H(s)\,.
\end{aligned}
\end{equation}
almost surely. 
\end{cor}

The question arises whether \eqref{eq.mildsolution} can be extended to hold for all $x^* 
\in E^*$. This is indeed the case under the following additional assumption.

\begin{hyp}\label{hyp2}
Assume Hypothesis \ref{hyp1}, that $S(t) \subset \cL (\tilde{E}, E)$ for all $t >0$ and that for 
$x \in \tilde{E}$ the $E$-valued map $t \mapsto S(t)x$ is continuous on $(0,\infty )$. 
Furthermore, assume that for all $t>0$ the function
$(0,t)\ni s \mapsto \| S(s)\|_{\cL (\tilde{E}, E)}$ is square integrable.
\end{hyp}

Assuming Hypothesis \ref{hyp2}, a slight variation of the arguments in Remark \ref{l.welldefined}
shows that in this case the integrals in \eqref{eq.mildsolution} are well-defined for $x^* \in E^*$.

\begin{cor}\label{c.extend}
Assume that Hypothesis \ref{hyp2} holds. If $\bX$ is a weak (equivalently, weakly mild) solution of \eqref{eq.sde},
then \eqref{eq.mildsolution} and \eqref{eq.stopped} hold for all $x^* \in E^*$.
\end{cor}

\begin{proof}
Define
\[ V := \{ x^* \in E^*\,:\,\eqref{eq.mildsolution} \,\, \mbox{holds a.e.} \,\}\,. \] 
By Proposition \ref{p.equivalence}, $\tilde{E}^* \subset V$ and hence $V$ is \ws-dense in $E^*$. 
The claim is proved once we show that $V$ is \ws-closed in $E^*$. By the Krein-Smulyan
theorem (see, e.g., \S 21.10 (6) of \cite{koethe}), $V$ is \ws-closed in $E^*$ if and only if 
 $B_{V} := \{ x^* \in V\,: \, \|x^*\|_{E^*} \leq 1\}$ is \ws-closed in $E^*$.
However, since the \ws-topology is metrizable on bounded sets, it suffices to prove that
$B_{V}$ is sequentially \ws-closed. 

Using Hypothesis \ref{hyp2}, this can be proved similarly as when extending 
equation \eqref{eq.solmild} from $x^* \in D((A^\odot )^2)$ to arbitrary $x^* \in \tilde{E}^*$ in 
the proof of Proposition \ref{p.equivalence}. The proof for \eqref{eq.stopped} is similar.
\end{proof}

\subsection{Mild solutions}\label{ssect.ms}
We begin by recalling some facts about stochastic integration of operator-valued processes. For time being, 
$B$ denotes a general separable Banach space and $H$ a separable Hilbert space. We also fix a stochastic basis 
$(\Omega, \Sigma, \FF, \P)$ satisfying the usual condition on which an $H$-cylindrical Wiener process with respect to 
$\FF$ is defined.

An \emph{elementary process} is a process $\Phi : [0,T]\times \Omega \to \cL (H,B)$ of the form
\[ \Phi (t,\omega ) = \sum_{n=1}^N\sum_{m=1}^M\one_{(t_{n-1}, t_n]\times A_{mn}}(t,\omega )\sum_{k=1}^K
 h_k \otimes x_{kmn} \,,
\]
where $0\leq t_0 < \cdots < t_N\leq T$, $A_{1n}, \cdots, A_{Mn} \in \cF_{t_{n-1}}$ are disjoint
for all $n$ and the vectors $h_1, \cdots, h_K$ are orthonormal in $H$. If $\Phi$ does not 
depend on $\omega$ we also say that $\Phi$ is an \emph{elementary function}.
For an elementary process, the 
\emph{stochastic integral} $\int_0^T\Phi (t)\, dW_H(t)$ is defined by
\[  \int_0^T\Phi (t) \, dW_H(t):= \sum_{n=1}^N\sum_{m=1}^M \one_{A_{mn}}
 \sum_{k=1}^K \big[W_H(t_n )h_k - W_H(t_{n-1})h_k\big]x_{kmn}
\]

Now let $\Phi : [0,T]\times \Omega \to \cL (H, B)$ be an $H$-strongly measurable and adapted process
which belongs to $L^2(0,T;H)$ scalarly, i.e.\ $\Phi^*x^* \in L^0(\Omega; L^2(0,T;H))$ for all $x^* \in B^*$.
Then $\Phi$ is called \emph{stochastically integrable} (on $(0,T)$) if there exists a sequence $\Phi_n$ of elementary
processes and an $C([0,T]; E)$-valued random variable $\eta$ such that
\begin{enumerate}
 \item $\dual{\Phi_n h}{x^*} \to \dual{\Phi h}{x^*}$ in $L^0(\Omega; L^2(0,T))$ for all $h \in H$ and
$x^*\in B^*$ and
 \item We have
\[ \eta (\cdot )= \lim_{n\to\infty} \int_0^\cdot \Phi_n(t)\, dW_H(t)\quad\mbox{in}\,\, L^0(\Omega;C([0,T];B))\,.\]
\end{enumerate}
In this case, $\eta$ is called the stochastic integral of $\Phi$ and we write
 $\int_0^t\Phi (t)\, dW_H(t) := \eta (t)$.
In the case where $\Phi$ does not depend on $\omega$, we also require that the approximating sequence $\Phi_n$ does
not depend on $\omega$.

Having defined stochastic integrability, we can now define what we mean by a \emph{mild solution}.
\begin{defn}
A tuple $((\Omega, \Sigma, \FF, \P),  W_H, \bX)$ where $(\Omega, \Sigma, \FF, \P)$ is stochastic basis satisfying the usual conditions,
$W_H$ is an $H$-cylindrical Wiener process with respect to $\FF$ and
$\bX$ is a continuous, $\FF$-progressive, $E$-valued process is called a \emph{mild solution} of \eqref{eq.sde} 
if for all $t \geq 0$ the function $s \mapsto S(t-s)G(X(s))$ is stochastically integrable and \eqref{eq.evalued}
holds almost surely.
\end{defn}
It is clear from the definition of stochastic integrability, that every mild solution of equation $[A,F,G]$ 
is also a weakly mild solution of $[A,F,G]$ and thus, by Proposition \ref{p.equivalence}, also a weak solution of $[A,F,G]$. 
Moreover, if $\bX$ is a mild solution, then \eqref{eq.mildsolution} even holds for all $x^* \in E^*$ 
(rather than for $x^* \in 
\tilde{E}^*$) and the exceptional set outside of which \eqref{eq.mildsolution} holds can be chosen independently of 
$x^*$. We also note that if $\bX$ is a weak (hence a weakly mild) solution and it is known a priori that $s \mapsto S(t-s)G(X(s))$ is stochastically 
integrable, then $\bX$ is a mild solution. 
\medskip 

The obvious question is whether for a weak solution $\bX$ the process $s \mapsto S_{t-s}G(X_s)$ is automatically stochastically integrable.
As we shall see, this is indeed the case in two important cases. The proof relies on a characterization of stochastic integrability of 
a process $\Phi$. Let us 
first discuss the case of $\cL (H,B)$-valued \emph{functions}, which was considered in \cite{vNW05}.
It was proved there that a function $\Phi : [0,T] \to \cL (H,B)$ is stochastically integrable if and only if
there exists an $B$-valued random variable $\xi$ such that
\begin{equation}\label{eq.xi}
 \dual{\xi}{x^*} = \int_0^T\Phi(s)^*x^* \, dW_H(s) .
\end{equation}
This, in turn, is equivalent with $\Phi$ representing a $\gamma$-Radonifying operator $R \in \gamma (L^2(0,T;H),B)$. 
For the definition of $\gamma$-Radonifying operators and more information, we refer to the survey article 
\cite{vNsurvey}. That
$\Phi$ \emph{represents} an operator $R \in \gamma (L^2(0,T;H),B)$ means that for all $x^* \in B^*$
the function $t \mapsto \Phi^*(t) x^*$ belongs to $L^2(0,T;H)$ and we have
\begin{equation}\label{eq.represent}
 \dual{Rf}{x^*} = \int_0^T[f(t), \Phi^*(t)x^*]_H\, dt \quad \forall\, f \in L^2(0,t;H)\,, \, x^* \in B^*.
\end{equation}
Note that if $\Phi$ is $H$-strongly measurable, then the operator $R$ is uniquely determined by $\Phi$.

Using the results of \cite{vNW05}, we obtain for \eqref{eq.sde} with additive noise:

\begin{prop}\label{p.additivecase}
Assume Hypotheses \ref{hyp1} and \ref{hyp2} and that $G \in \cL(H, \tilde{E})$ is constant.
Then the weak, the weakly mild and the mild solutions of \eqref{eq.sde} coincide. Furthermore, 
if there exist solutions,
the function $s \mapsto S_{t-s}G$ represents an element of $\gamma (L^2(0,t;H),E)$ for all $t>0$.
\end{prop}

\begin{proof}
Let $\bX$ be a weak (equivalently, a weakly mild) solution of \eqref{eq.sde}. If no such solution 
exists, there is nothing to prove since every mild solution is also a weakly mild solution.

Arguing as Remark \ref{l.welldefined}, using that as a consequence of Hypothesis \ref{hyp2} 
the map $s \mapsto \dual{x}{S_{t-s}^*x^*}$ is continuous even for $x^* \in E^*$ and $x \in \tilde{E}$, 
we see that $(s,\omega ) \mapsto \dual{S(t-s)F(X(s,\omega ))}{x^*}$
is measurable for all $x^* \in E^*$. By Hypothesis \ref{hyp2}, $\norm{S_s}_{\cL (\tilde{E}, E)}$ is majorized
on $(0,t)$ by a square integrable function, say $g$. Hence, by the boundedness of $F$ on bounded sets we have
\[ \norm{S_{t-s}F(X(s,\omega ))} \leq g(t-s)\sup_{r \in (0,t)}\norm{F(X(r,\omega)} \in L^1(0,t) .\]
This implies that $\int_0^t S_{t-s}F(X_s)\, ds$ can be defined pathwise as an $E$-valued Bochner integral. 
Furthermore, this integral is a weakly measurable function of $\omega$. 
Since $E$ is separable, $\int_0^t S_{t-s}F(X_s)\, ds$ is a strongly measurable function of $\omega$ by 
the Pettis measurability theorem. Consequently, $\xi := X_t - S_tX_0 - \int_0^t S_{t-s}F(X_s)\, ds$ is 
an $E$-valued random variable. Since $\bX$ is a weakly mild solution, \eqref{eq.xi} holds 
for $T := t, \Phi : s \mapsto S_{t-s}G$ and all  $x^* \in E^*$ by Corollary \ref{c.extend}. 
The claim follows from the results of \cite{vNW05}.
\end{proof}

Let us now return to our discussion of stochastic integrability in a general separable Banach space
$B$. In order to have a powerful integration theory for $\cL (H,B)$-valued \emph{processes}, we need 
an additional geometric assumption on $B$. Of particular importance are the so-called \emph{UMD Banach spaces}.
For the definition of UMD spaces and more information, we refer to the survey article \cite{burk}. 
We here confine ourselves to note that
every Hilbert space is a UMD space as are the reflexive $L^p$ and Sobolev spaces. 

The importance of the UMD property for stochastic integration is that it allows for so-called decoupling, see 
\cite{garling, mcc}. Roughly speaking, this enables us to replace the cylindrical Wiener process $W_H$
by an independent copy $\tilde{W}_H$ and thus use the results of \cite{vNW05} pathwise. This program was carried
out in \cite{vNVW07} and yields a similar characterization of stochastic integrability as in \cite{vNW05}
in the case of processes which belong scalarly to $L^p(\Omega; L^2(0,T;H))$. We recall that $\Phi: [0,T]\times \Omega \to \cL (H, E))$ is said to 
\emph{belong to} $L^p(\Omega; L^2(0,T;H))$ \emph{scalarly}, if for every $x^* \in E^*$  the function
$t \mapsto \Phi^*(t, \omega)x^*$ belongs to $L^2(0,T;H)$ for almost every $\omega$ and the map $\omega \mapsto \Phi^*(\cdot, \omega)x^*$ belongs to $L^p(\Omega; L^2(0,T;H))$.

It is proved in \cite{vNVW07}
that an $H$-strongly measurable and adapted process $\Phi : [0,T]\times \Omega \to \cL (H,E)$ which
belongs to $L^p(\Omega; L^2(0,T;H))$ scalarly is stochastically integrable if and only if there is a 
random variable $\xi \in L^p(\Omega; E)$ such that \eqref{eq.xi} holds
for all $x^* \in E^*$. 
This in turn is the case if and only if $\Phi$ represents a random variable $R \in L^p(\Omega; 
\gamma (L^2(0,T;H),E))$. Here `represents' means that \eqref{eq.represent} holds for almost every $\omega$.

A characterization of stochastic integrability 
for processes $\Phi$ which belong scalarly to $L^0(\Omega; L^2(0,T;H))$ is also contained in \cite{vNVW07}, 
however, in this characterization one needs information about the whole integral process $\int_0^\cdot 
\Phi (s) dW_H(s)$; when dealing with weakly mild solutions, such information is not available, whence this 
characterization cannot be used for our purposes. 
Therefore, in the proposition below, we use a stopping time argument to reduce to the $L^p(\Omega)$-case.

\begin{prop}\label{p.umd}
Assume Hypotheses \ref{hyp1} and \ref{hyp2} and that $E$ is a UMD Banach space. Then the weak, the weakly mild and 
the mild solutions of \eqref{eq.sde} coincide. Furthermore, if $\bX$ is a weak solution, then for all $t \geq 0$ 
the function $s \mapsto S_{t-s}G(X_s)$ represents an element of the space $L^0(\Omega, \gamma (L^2(0,t;H),E))$.
\end{prop}

\begin{proof}
Let $\bX$ be a weak (equivalently, a weakly mild) solution of \eqref{eq.sde}. If no weak solution exists, there 
is nothing to prove.

For $n \in \CN$ and define $\tau_n := \inf\{s>0\,:\, \norm{X_s} \geq n \}$. Fix $t >0$. 
Arguing similar as in the proof of Proposition \ref{p.additivecase}, we see that
\[ \xi_n := X_{t\wedge \tau_n} -(X_{\tau_n} - S_{t-t\wedge \tau_n}X_{\tau_n})\one_{\{\tau_n < \infty\}}- S_t X_{0\wedge \tau_n} - \int_0^t\one_{[0,\tau_n]}S_{t-s}F(X_s)\, ds\]
is a well-defined, bounded, $E$-valued random variable. 
It follows from Corollary \ref{c.extend}, that for $x^* \in E^*$,  
\[ \dual{\xi_n}{x^*} = \int_0^t \one_{[0,\tau_n]} G(X_s)^*S_{t-s}^*x^*\, dW_H(s) \,.\]
almost surely.
Since $\bX$ has continuous paths and $G$ is bounded on bounded subsets,
$\Phi_n : s \mapsto \one_{[0,\tau_n]}S_{t-s}G(X_s)$ belongs to $L^\infty (\Omega; L^2(0,t;H))$ scalarly. Hence,
by \cite[Theorem 5.9]{vNVW07}, $\Phi_n$ is stochastically integrable and 
\begin{equation}\label{eq.stoppedequation}
\begin{aligned}
X_{t\wedge \tau_n} & =  \,\, S_tX_{0\wedge \tau_n } + X_{\tau_n} - S_{t-t\wedge \tau_n}X_{\tau_n}\\
& \quad +
\int_0^{t\wedge \tau_n}  S_{t-s}F(X_s)\, ds
  + \int_0^t\one_{[0,\tau_n]}S_{t-s}G(X_s)\, dW_H(s)\,.
  \end{aligned}
\end{equation}
Furthermore, $\Phi_n$ represents an element of $L^p(\Omega; \gamma (L^2(0,t;H), E))$ for all $p\geq 1$.
Now let $N$ be a set with $\P(N) = 0$ such that for $\omega \not\in N$ the map $s \mapsto \Phi_n(s, \omega )$
represents an element $R_n(\omega )$ of $\gamma (L^2(0,t;H),E)$. Such a set exists by \cite[Lemma 2.7]{vNVW07}. 

Note that by the continuity of the paths, $\Phi_n (s, \omega ) = \Phi (s,\omega ):= S_{t-s}G(X(s, \omega ))$ 
for all $s \in (0,t)$ and $n\geq n_0=n_0(\omega )$. Thus, $\Phi (s, \omega )$ represents an element $R(\omega )$ of
$\gamma (L^2(0,t;H),E)$ for all $\omega \not \in N$. Since $R_n(\omega ) \to R(\omega )$ for all $\omega \not
\in N$, it follows that $R$ is a strongly measurable $\gamma (L^2(0,t;H),E)$-valued random variable. Furthermore,
$R$ is represented by $\Phi$. By \cite[Theorem 5.9]{vNVW07}, $\Phi$ is stochastically integrable and 
\cite[Theorem 5.5]{vNVW07} shows that
\[ \int_0^t\Phi_n(s)\, dW_H(s) \to \int_0^t \Phi (s)\, dW_H(s) \quad \mbox{in} \,\, L^0(\Omega;E)\,.\]
On the other hand, 
\[ \xi_n \to X(t) - S(t)X(0) - \int_0^t S(t-s)F(X(s))\, ds \]
pointwise a.e.\ and hence in $L^0(\Omega; E)$. Thus, letting $n \to \infty$ in \eqref{eq.stoppedequation}
finishes the proof.
\end{proof}

\section{Applications}\label{sect.applications}

We end this article by discussing some examples of stochastic partial differential equations where the results of this article can be applied.

\subsection{Equations with measurable semilinear term and additive noise}
In \cite{k13}, we are concerned with the following equation
\begin{equation}\label{eq.ou}
dX(t) = \big[AX(t) + F(X(t)\big] + GdW_H(t)
\end{equation}
where $E, \tilde{E}, H$ and $A$ are as in Hypothesis \ref{hyp1}, the semilinear term $F: E \to E$ is bounded and measurable, 
$W_H$ is an $H$-cylindrical Wiener process and $G \in \cL (H, \tilde{E})$. In the case where $F \equiv 0$, this is an Ornstein-Uhlenbeck equation, which is well understood. If the Ornstein-Uhlenbeck equation associated with \eqref{eq.ou}, i.e.\ equation $[A, 0, G]$ is well-posed, the associated transition semigroup
$\mathscr{T}_\mathsf{ou}$ is known explicitly. Namely,
\[
\mathscr{T}_\mathsf{ou}(t)f(x) = \int_E f(S(t)x +y)\, d\mathscr{N}_{Q_t}(y)
\]
where $\mathscr{N}_Q$ denotes the centered Gaussian measure with covariance operator $Q$ and $Q_t : E^* \to E$ is given as
\[
Q_tx^* := \int_0^t S(s)GG^*S(s)^*x^*\, ds.
\]
By $H_{Q_t}$, we denote the reproducing kernel Hilbert space associated with $Q_t$. In \cite{k13}, the following theorem is proved.
\begin{thm}\label{t.ou}
Let $E, \tilde{E}, H$ and $A$ as in Hypothesis \ref{hyp1}, $G \in \cL (H, \tilde{E})$  and assume that also Hypothesis \ref{hyp2} is satisfied. Moreover, assume that the Ornstein-Uhlenbeck equation $[A, 0, G]$ is well-posed 
and that $S(t)E \subset H_{Q_t}$ for all $t> 0$ with 
\begin{equation}\label{eq.condition}
\int_0^T\|S(t)\|_{\cL (E, H_{Q_t})} \, dt < \infty
\end{equation}
for all $T> 0$. Then for every bounded, measurable $F: E \to E$ equation \eqref{eq.ou} is well-posed. The solutions are strong Markov processes
with a strong Feller transition semigroup.
\end{thm}

This extends earlier results from \cite{cmg95, gg94, gg94a} where the corresponding equation was studied for bounded 
and continuous (resp.\ bounded and weakly continuous) $F$ under similar assumptions in the case where $E=\tilde{E}$ is a Hilbert space. 
The assertion that \eqref{eq.ou} is well-posed even for bounded \emph{measurable} $F$ appears to be new even in the case of Hilbert spaces since existence of solutions cannot be inferred from the Girsanov theorem, as $G$ is, in general, not invertible.

The assumption that \eqref{eq.condition} holds 
implies that the transition semigroup $\mathscr{T}_\mathsf{ou}$ is strongly Feller and is satisfied in many important examples, for example 
for the one-dimensional stochastic heat equation driven by space-time white noise, i.e.\ $A$ is the $L^p$-realization of the Dirichlet Laplacian on the interval $(0,1)$ and for $p\leq 2$ we set the operator $G$ is the injection from $L^2(0,1)$ to $L^p(0,1)$. In the case $p>2$ we set $\tilde{E} = L^2(0,1)$ and $G$ the identity. It is also possible to consider the stochastic heat equation on $C_0(0,1)$. More examples, which include equations in higher space 
dimension, more general differential operators and different noise terms are discussed in \cite{k13}.
\medskip

The proof of Theorem \ref{t.ou} is based on Theorem \ref{t.weakmart}, and we prove existence and uniqueness of solutions of the associated 
local martingale problem. The actual proof of existence and uniqueness is then given using semigroup theory.
In view of Theorem \ref{t.wellposed}, the strong Markov property for solutions follows 
automatically once we have established well-posedness of $[A, F, G]$. 

The first step to prove uniqueness for solutions of \eqref{eq.ou} is to prove a Miyadera-Voigt type perturbation result for strongly Feller semigroups. For the generator $\mathscr{A}_\mathsf{ou}$ of the Ornstein-Uhlenbeck semigroup $\mathscr{T}_\mathsf{ou}$, this result can be used to show that $\mathscr{A}_\mathsf{pert}$, defined 
by $\mathscr{A}_\mathsf{pert}u(x) := \mathscr{A}_\mathsf{ou}u(x) +
\langle F(x), \nabla u(x)\rangle$, generates a strongly Feller semigroup $\mathscr{T}_\mathsf{pert}$. A detailed analysis of the 
operator $\mathscr{A}_\mathsf{pert}$ shows that a probability measure $\PP$ on $C([0,\infty); E)$ solves the local martingale problem 
associated with equation $[A,F,G]$ if and only if it solves the true martingale problem (in the sense of \cite{ek}) for the operator $\mathscr{A}_\mathsf{pert}$. Thus a well-known result \cite[Theorem 4.4.1]{ek} yields that the one-dimensional distributions 
of a solution $\PP$ of the martingale problem for $\mathscr{A}_\mathsf{pert}$ are determined by the distribution of $\bx (0)$ under $\PP$ 
and the semigroup $\mathscr{T}_\mathsf{pert}$. By Theorem \ref{t.markov}, this implies uniqueness in law for the solutions 
of equation \eqref{eq.ou}. Moreover, if solutions exist, then the associated transition semigroup is $\mathscr{T}_\mathsf{pert}$, which is 
strongly Feller.

It thus remains to prove existence of solutions. If $F$ is additionally Lipschitz continuous, then solutions can be constructed using Banach's fixed point 
theorem in a standard way. Thus, for bounded, Lipschitz continuous $F$, equation \eqref{eq.ou} is well-posed.
To extend the existence result to general bounded, measurable $F$, a refinement of Lemma \ref{l.approx} is used. 
Indeed, making use of the strong Feller property, it can be proved that if $F_n$ is a sequence of bounded measurable functions such that 
equation $[A, F_n, G]$ is well-posed for every $n$ and the sequence $F_n$ is uniformly bounded and converges pointwise to the bounded function 
$F$, then also equation $[A, F, G]$ is well-posed. The tightness of the solutions to the local martingale problem for $[A, F_n, G]$ can be proved 
using that these measures are distributions of mild solutions of the equation.
Using the approximation result, well-posedness of \eqref{eq.ou} can be extended from bounded, Lipschitz 
continuous $F$ to bounded, measurable $F$ via a monotone class argument.

\subsection{Stochastic reaction-diffusion systems with H\"older continuous multiplicative noise}

Reaction-diffusion systems and stochastic perturbations of them play an important role in applications in chemistry, biology and physics \cite{murray}. In an 
abstract form, a stochastic reaction-diffusion system takes the form \eqref{eq.sde}, where the state space $E$ is a Banach space of $\CR^r$-valued 
functions, defined on a domain $\OO\subset \CR^d$. Typically, the reaction term $F$ is a vector of composition operators with polynomial entries.

Such systems with locally Lipschitz continuous multiplicative noise where studied in \cite{Cerrai}. In the case where the noise term $G$ is merely 
H\"older continuous, only partial results are available and, to the best of our knowledge, only for $r=1$, i.e.\ a single reaction-diffusion equation rather 
than a system. In \cite{bg99}, existence of solutions for such an equation was proved under an additional boundedness assumption on $G$. However, 
a uniqueness result is missing, except for the case of locally Lipschitz continuous $G$.\smallskip

In \cite{k12}, we prove pathwise uniqueness and strong existence of solutions for a class of stochastic reaction-diffusion equations with 
H\"older continuous multiplicative noise. Let us here present an example which fits into the framework of \cite{k12} and explain how results 
of this article are used in the proof of existence and uniqueness.

Let $\OO \subset \CR^d$ be an open domain with Lipschitz boundary.
Moreover, we let $a_1 = (a_{ij}^{(1)}), a_2=(a_{ij}^{(2)}) \in L^\infty(\OO; \CR^{d\times d})$ be symmetric and uniformly elliptic, i.e.\ there exists $\eta> 0$ such that 
for all $\xi \in \CR^d$ and almost all $x \in \OO$ we have
\[
\sum_{i,j=1}^da_{ij}^{(l)}(x) \xi_i\xi_j \geq \eta |\xi|^2
\]
for $l=1,2$. Let $R_1, R_2$ be Hilbert-Schmidt operators on $L^2(\OO)$ such that 
$R_j$ is diagonalized by an orthonormal basis $(e_n^{(j)})_{n\in \CN}$ of $L^2(\OO)$ which consists of functions in $C(\overline{\OO})$ 
and satisfies $\sum_{n=1}^\infty \|R_je_n^{(j)}\|_\infty^2 < \infty$ for $j=1,2$. Finally, we let $g_1, g_2: \CR \to \CR$ be of linear growth and locally $\half$-H\"older continuous. We consider the following stochastic reaction-diffusion system
\begin{equation}\label{eq.rds}
\left\{
\begin{array}{lll}
du_1(t) & = & \big[\div (a_1\nabla u_1(t)) +u_1(t) - u_1(t)^3 + u_2(t)\big]dt + g_1(u_1(t))R_1dW_1(t)\\
du_2(t) & = & \big[\div (a_2\nabla u_2(t)) +u_1(t) - u_2(t)\big]dt + g_2(u_2(t))R_2dW_2(t)
\end{array}
\right.
\end{equation}
complemented with conormal boundary conditions. 

To reformulate the above system in our abstract framework, we set $\tilde E =E= C(\overline{\OO})\times C(\overline{\OO})$ and $A = \mathrm{diag}(A_1, A_2)$,
where $A_j$ is the $C(\overline{\OO})$-realization of the differential operator $\div (a_j \nabla \cdot)$ under conormal boundary conditions. 
We set $H = L^2(\OO)\times L^2(\OO)$. By the assumption on $R_j$, for $h \in L^2(\OO)$ we find that $R_j h \in C(\overline{\OO})$. We may
thus define $G: E \to \cL (H,E)$ by 
\[ [G(u,v)h](x) := ( g_1(u(x))R_1h_1(x), g_2(u(x))R_2h_2(x))\]
  for $h_1,h_2 \in L^2(\OO)$ and $x \in \overline{\OO}$. 
The reaction term $F$ is given by $[F(u,v)](x) := (u(x)-u(x)^3+v(x), u(x) -v(x))$. This reaction 
Term is of Fitzhugh-Nagumo type and equations with this reaction term are generic excitable systems \cite{murray}.

In \cite{k12} we prove

\begin{thm}\label{t.rds}
Under the assumptions above, equation \eqref{eq.rds} is well-posed on the state space $E = C(\overline{\OO})\times C(\overline{\OO})$.
The solutions exist strongly, they are pathwise unique and strong Markov processes.
\end{thm}

The proof of Theorem \ref{t.rds} is in spirit rather different from the proof of well-posedness of \eqref{eq.ou}, insofar as we work directly with 
solutions of the equation, rather than with solutions of the associated local martingale problem.  In the proof, we use the equivalence of weak and 
mild solutions. Indeed, in the proof of pathwise uniqueness, we use weak solutions, whereas in the proof of existence of solutions, we use 
mild solutions. We also employ the Yamada-Watanabe theory from  
Section \ref{sect.yw}.

The proof of pathwise uniqueness is an adaption of the proof of \cite[Theorem 1]{yw1}. The main difficulty in extending the proof from the finite-dimensional setting to an infinite dimensional setting is to handle the differential operators involved in \eqref{eq.rds}. In \cite{k12}, we use 
the concept of a weak solution and test solutions against functionals $x^* = (\lambda R(\lambda, A_1)^*\delta_x, 0)$, resp.\ 
$x^* = (0, \lambda R(\lambda, A_2)^*\delta_x)$, where $A_j$ are the realizations of of the differential operator $\div (a_j\nabla \cdot)$ on 
$C(\overline{\OO})$. This approach should be compared with \cite{mps06}, where pathwise uniqueness was proved for stochastic heat equations 
on $\OO = \CR^d$, namely
\[
du(t) = \Delta u(t) + \sigma (u(t))dW(t),
\]
where $\Delta$ is the Laplacian on $\CR^d$, $W$ is a colored noise and $\sigma : \CR \to \CR$ is $\gamma$-H\"older continuous, where the 
allowed value of $\gamma$ depends on the noise $W$. To prove pathwise uniqueness in \cite{mps06}, the authors convolute solutions of 
the stochastic heat equation with a mollifier $\varphi_n$. In their variational framework, this yields the term $u\ast \Delta\varphi_n$ in the equation 
for the resulting process. It is then used that, as a consequence of its translation invariance, the Laplacian commutes with convolutions, i.e.\ 
we have $u\ast (\Delta \varphi_n) = \Delta (u\ast\varphi_n)$. This is no longer true for differential operators with nonconstant coefficients
as in \eqref{eq.rds}. 

Let us also note that a recent result \cite{mmp12} for the stochastic heat equation that in the case of $d=1$ shows that we cannot hope for pathwise uniqueness in the case of space-time 
white noise.\smallskip

Note that by Theorem \ref{t.yw}, pathwise uniqueness implies uniqueness in law, hence the strong Markov property of solutions follows 
from Theorem \ref{t.markov} once we have established existence of solutions.
To that end, we approximate the function $f$ in the reaction term and the functions $g_1, g_2$ with bounded functions by cutting 
off the functions.
Existence of solutions for the approximate problems with bounded coefficients and deterministic initial values 
follows from the results of \cite{bg99}. We could then 
use Lemma \ref{l.approx} to infer existence of solutions for the limit problem \eqref{eq.rds}. However, in \cite{k12} we choose a different approach
and use that, as a consequence of pathwise uniqueness and Corollary \ref{c.strongexistence}, the approximate solutions can be realized on a common 
stochastic basis and with respect to a common $H$-cylindrical Wiener process. This allows us to adopt the strategy from \cite{Cerrai, KvN11} to prove 
existence of solutions. Indeed, as the approximate solutions exist on a common stochastic basis and are pathwise unique, they can be `glued together' 
to a `maximal solution' of equation \eqref{eq.rds}. To prove existence of solutions in the sense used here, we have to prove that the `maximal solution' 
exists globally. By the results of \cite{KvN11}, to that end, we have to prove uniform boundedness of the approximate solutions in $L^p(\Omega; C([0,T]; E))$ for a suitable $p>1$, all $T>0$ and $p$-integrable initial data. As the approximate solutions are also \emph{mild} solutions, the uniform boundedness can be proved using estimates for deterministic and stochastic convolutions, see \cite{vNVW08}.

We note that, in comparison with \cite{bg99}, in Theorem \ref{t.rds} we do not need that the term $G$ is bounded. Moreover, with the above arguments, 
we initially prove existence of solutions only for initial data with a certain integrability, thus in particular for deterministic initial data. However, 
by Theorem \ref{t.markov}, we automatically obtain existence of solutions for all initial distributions.

\subsection*{Acknowledgment} I would like to thank Jan van Neerven for several helpful discussions and also for reading an earlier version of this article.
I am also grateful to the anonymous referees for the critical comments, which helped improve this article.

\providecommand{\bysame}{\leavevmode\hbox to3em{\hrulefill}\thinspace}
\providecommand{\MR}{\relax\ifhmode\unskip\space\fi MR }
\providecommand{\MRhref}[2]{%
  \href{http://www.ams.org/mathscinet-getitem?mr=#1}{#2}
}
\providecommand{\href}[2]{#2}


\begin{thebibliography}{10}

\bibitem{bogachev}
V.~I. Bogachev, \emph{Measure theory. {V}ol. {I}, {II}}, Springer-Verlag,
  Berlin, 2007.

\bibitem{bg99}
Z.~Brze{\'z}niak and D.~G{\c{a}}tarek, \emph{Martingale solutions and invariant
  measures for stochastic evolution equations in {B}anach spaces}, Stochastic
  Process. Appl. \textbf{84} (1999), no.~2, 187--225.

\bibitem{bms05}
Z.~Brze{\'z}niak, B.~Maslowski, and J.~Seidler, \emph{Stochastic nonlinear beam
  equations}, Probab. Theory Related Fields \textbf{132} (2005), no.~1,
  119--149.

\bibitem{burk}
D.~L. Burkholder, \emph{Martingales and singular integrals in {B}anach spaces},
  Handbook of the geometry of {B}anach spaces, {V}ol. {I}, North-Holland,
  Amsterdam, 2001, pp.~233--269.

\bibitem{Cerrai}
S.~Cerrai, \emph{Stochastic reaction-diffusion systems with multiplicative
  noise and non-{L}ipschitz reaction term}, Probab. Theory Related Fields
  \textbf{125} (2003), no.~2, 271--304.

\bibitem{cmg95}
A.~Chojnowska-Michalik and B.~Go{\l}dys, \emph{Existence, uniqueness and
  invariant measures for stochastic semilinear equations on {H}ilbert spaces},
  Probab. Theory Related Fields \textbf{102} (1995), no.~3, 331--356.

\bibitem{dpkz}
G.~Da~Prato, S.~Kwapie{\'n}, and J.~Zabczyk, \emph{Regularity of solutions of
  linear stochastic equations in {H}ilbert spaces}, Stochastics \textbf{23}
  (1987), no.~1, 1--23.

\bibitem{dpz}
G.~Da~Prato and J.~Zabczyk, \emph{Stochastic equations in infinite dimensions},
  Encyclopedia of Mathematics and its Applications, vol.~44, Cambridge
  University Press, Cambridge, 1992.

\bibitem{ek}
S.~N. Ethier and T.~G. Kurtz, \emph{Markov processes}, Wiley Series in
  Probability and Mathematical Statistics: Probability and Mathematical
  Statistics, John Wiley \& Sons Inc., New York, 1986, Characterization and
  convergence.

\bibitem{garling}
D.~J.~H. Garling, \emph{Brownian motion and {UMD}-spaces}, Probability and
  {B}anach spaces ({Z}aragoza, 1985), Lecture Notes in Math., vol. 1221,
  Springer, Berlin, 1986, pp.~36--49.

\bibitem{gg94}
D.~G{\c{a}}tarek and B.~Go{\l}dys, \emph{On uniqueness in law of solutions to
  stochastic evolution equations in {H}ilbert spaces}, Stochastic Anal. Appl.
  \textbf{12} (1994), no.~2, 193--203.

\bibitem{gg94a}
Dariusz G{\c{a}}tarek and Beniamin Go{\l}dys, \emph{On weak solutions of
  stochastic equations in {H}ilbert spaces}, Stochastics Stochastics Rep.
  \textbf{46} (1994), no.~1-2, 41--51.

\bibitem{hs12}
M.~Hofmanov{\'a} and J.~Seidler, \emph{On weak solutions of stochastic
  differential equations}, Stoch. Anal. Appl. \textbf{30} (2012), no.~1,
  100--121.

\bibitem{iw}
N.~Ikeda and S.~Watanabe, \emph{Stochastic differential equations and diffusion
  processes}, second ed., North-Holland Mathematical Library, vol.~24,
  North-Holland Publishing Co., Amsterdam, 1989.

\bibitem{js}
J.~Jacod and A.~N. Shiryaev, \emph{Limit theorems for stochastic processes},
  second ed., Grundlehren der Mathematischen Wissenschaften [Fundamental
  Principles of Mathematical Sciences], vol. 288, Springer-Verlag, Berlin,
  2003.

\bibitem{kallenberg}
O.~Kallenberg, \emph{Foundations of modern probability}, second ed.,
  Probability and its Applications (New York), Springer-Verlag, New York, 2002.

\bibitem{ks}
I.~Karatzas and S.~E. Shreve, \emph{Brownian motion and stochastic calculus},
  second ed., Graduate Texts in Mathematics, vol. 113, Springer-Verlag, New
  York, 1991.

\bibitem{koethe}
G.~K{\"o}the, \emph{Topological vector spaces. {I}}, Translated from the German
  by D. J. H. Garling. Die Grundlehren der mathematischen Wissenschaften, Band
  159, Springer-Verlag New York Inc., New York, 1969.

\bibitem{k12}
M.~C. Kunze, \emph{Stochastic reaction-diffusion systems with {H}\"older
  continuous multiplicative noise}, preprint. arXiv:1209.4821, 2012.

\bibitem{k13}
M.~C. Kunze, \emph{Perturbation of strong feller semigroups and well-posedness
  of semilinear stochastic equations on banach spaces}, Stochastics An
  International Journal of Probability and Stochastic Processes \textbf{85}
  (2013), no.~6, 960--986.

\bibitem{KvN11}
M.~C. Kunze and J.~M.~A.~M. van Neerven, \emph{Continuous dependence on the
  coefficients and global existence for stochastic reaction diffusion
  equations}, J. Differential Equations \textbf{253} (2012), no.~3, 1036--1068.

\bibitem{kurtz07}
T.~G. Kurtz, \emph{The {Y}amada-{W}atanabe-{E}ngelbert theorem for general
  stochastic equations and inequalities}, Electron. J. Probab. \textbf{12}
  (2007), 951--965.

\bibitem{mcc}
Terry~R. McConnell, \emph{Decoupling and stochastic integration in {UMD}
  {B}anach spaces}, Probab. Math. Statist. \textbf{10} (1989), no.~2, 283--295.

\bibitem{mmp12}
C.~Mueller, L.~Mytnik, and E.~Perkins, \emph{Nonuniqueness for a parabolic
  {SPDE} with $\frac{3}{4}-\eps$-{H}\"older diffusion coefficients}, preprint.
  arXiv:1201.2767, 2012.

\bibitem{murray}
J.~D. Murray, \emph{Mathematical biology}, second ed., Biomathematics, vol.~19,
  Springer-Verlag, Berlin, 1993.

\bibitem{mps06}
Leonid Mytnik, Edwin Perkins, and Anja Sturm, \emph{On pathwise uniqueness for
  stochastic heat equations with non-{L}ipschitz coefficients}, Ann. Probab.
  \textbf{34} (2006), no.~5, 1910--1959.

\bibitem{vNadjoint}
J.~M. A.~M. {\noopsort{Neerven}}{van Neerven}, \emph{The adjoint of a semigroup
  of linear operators}, Lecture Notes in Mathematics, vol. 1529,
  Springer-Verlag, Berlin, 1992.

\bibitem{vNsurvey}
\bysame, \emph{$\gamma$-{R}adonifying operators: a survey}, AMSI-ANU Workshop
  on Spectral Theory and Harmonic Analysis, 2010, pp.~1--62.

\bibitem{vNV06}
J.~M. A.~M. {\noopsort{Neerven}}{van Neerven} and M.~C. Veraar, \emph{On the
  stochastic {F}ubini theorem in infinite dimensions}, Stochastic partial
  differential equations and applications---{VII}, Lect. Notes Pure Appl.
  Math., vol. 245, Chapman \& Hall/CRC, Boca Raton, FL, 2006, pp.~323--336.

\bibitem{vNVW07}
J.~M. A.~M. {\noopsort{Neerven}}{van Neerven}, M.~C. Veraar, and L.~Weis,
  \emph{Stochastic integration in {UMD} {B}anach spaces}, Ann. Probab.
  \textbf{35} (2007), no.~4, 1438--1478.

\bibitem{vNVW08}
\bysame, \emph{Stochastic evolution equations in {UMD} {B}anach spaces}, J.
  Funct. Anal. \textbf{255} (2008), no.~4, 940--993.

\bibitem{vNW05}
J.~M. A.~M. {\noopsort{Neerven}}{van Neerven} and L.~Weis, \emph{Stochastic
  integration of functions with values in a {B}anach space}, Studia Math.
  \textbf{166} (2005), no.~2, 131--170.

\bibitem{onddiss}
M.~Ondrej{\'a}t, \emph{Uniqueness for stochastic evolution equations in
  {B}anach spaces}, Dissertationes Math. (Rozprawy Mat.) \textbf{426} (2004),
  63.

\bibitem{ond05}
\bysame, \emph{Brownian representations of cylindrical local martingales,
  martingale problem and strong {M}arkov property of weak solutions of {SPDE}s
  in {B}anach spaces}, Czechoslovak Math. J. \textbf{55(130)} (2005), no.~4,
  1003--1039.

\bibitem{ondr1}
\bysame, \emph{Integral representations of cylindrical local martingales in
  every separable {B}anach space}, Infin. Dimens. Anal. Quantum Probab. Relat.
  Top. \textbf{10} (2007), no.~3, 365--379.

\bibitem{partha}
K.~R. Parthasarathy, \emph{Probability measures on metric spaces}, AMS Chelsea
  Publishing, Providence, RI, 2005, Reprint of the 1967 original.

\bibitem{sv10}
R.~Schnaubelt and M.~C. Veraar, \emph{Structurally damped plate and wave
  equations with random point force in arbitrary space dimensions},
  Differential Integral Equations \textbf{23} (2010), no.~9-10, 957--988.

\bibitem{sv1}
D.~W. Stroock and S.~R.~S. Varadhan, \emph{Diffusion processes with continuous
  coefficients. {I} and {II}}, Comm. Pure Appl. Math. \textbf{22} (1969),
  345--400 and 479--530.

\bibitem{yw1}
T.~Yamada and S.~Watanabe, \emph{On the uniqueness of solutions of stochastic
  differential equations.}, J. Math. Kyoto Univ. \textbf{11} (1971), 155--167.

\bibitem{zambotti}
L.~Zambotti, \emph{An analytic approach to existence and uniqueness for
  martingale problems in infinite dimensions}, Probab. Theory Related Fields
  \textbf{118} (2000), no.~2, 147--168.

\bibitem{zimmer}
J.~Zimmerschied, \emph{\"uber eine {F}aktorisierungsmethode f\"ur stochastische
  {E}volutionsgleichungen in {B}anachr\"aumen}, Ph.D. thesis, Universi\"at
  Karlsruhe, 2006.

\end{thebibliography}
\end{document}